\documentclass{amsart}

\usepackage{amssymb}
\usepackage{amsmath}
\usepackage{amsthm}
\usepackage{graphicx}

\newcommand{\erase}[1]{}

\newtheorem{theorem}{Theorem}[section]
\newtheorem{lemma}[theorem]{Lemma}
\newtheorem{proposition}[theorem]{Proposition}
\newtheorem{corollary}[theorem]{Corollary}

\newtheorem{_definition}[theorem]{Definition}
\newenvironment{definition}{\begin{_definition}\rm}{\end{_definition}}

\newtheorem{_remark}[theorem]{\it Remark}
\newenvironment{remark}{\begin{_remark}\rm}{\end{_remark}}

\newtheorem{_example}[theorem]{Example}
\newenvironment{example}{\begin{_example}\rm}{\end{_example}}

\newcommand{\F}{\mathord{\mathbb F}}
\newcommand{\HH}{\mathord{\mathbb H}}
\renewcommand{\P}{\mathord{\mathbb  P}}
\newcommand{\Q}{\mathord{\mathbb  Q}}
\newcommand{\R}{\mathord{\mathbb R}}
\newcommand{\Z}{\mathord{\mathbb Z}}

\newcommand{\BBB}{\mathord{\mathcal B}}
\newcommand{\EEE}{\mathord{\mathcal E}}

\newcommand{\HHH}{\mathord{\mathcal H}}
\newcommand{\LLL}{\mathord{\mathcal L}}
\newcommand{\OOO}{\mathord{\mathcal O}}
\newcommand{\PPP}{\mathord{\mathcal P}}
\newcommand{\QQQ}{\mathord{\mathcal Q}}
\newcommand{\RRR}{\mathord{\mathcal R}}
\newcommand{\TTT}{\mathord{\mathcal T}}
\newcommand{\VVV}{\mathord{\mathcal V}}

\newcommand{\maprightsp}[1]{\; \smash{\mathop{\; \longrightarrow \; }\limits\sp{#1}}\; }

\newcommand{\inj}{\hookrightarrow}
\newcommand{\isom}{\mathbin{\,\raise -.6pt\rlap{$\to$}\raise 3.5pt \hbox{\hskip .3pt$\mathord{\sim}$}\,}}

\newcommand{\set}[2]{\{\, {#1} \, \mid \, {#2} \,  \}}
\newcommand{\shortset}[2]{\{ {#1} \,|\, {#2}   \}}

\newcommand{\gen}[1]{\langle {#1}  \rangle}

\newcommand{\tensor}{\otimes}

\newcommand{\sprime}{\sp\prime}

\newcommand{\spar}[1]{\sp{(#1)}}
\newcommand{\spprime}{\sp{\prime\prime}}

\newcommand{\sptimes}{\sp{\times}}
\newcommand{\sperp}{\sp{\perp}}

\newcommand{\inv}{\sp{-1}}

\newcommand{\NS}{\mathord{\mathrm {NS}}}

\newcommand{\GL}{\mathord{\mathrm {GL}}}

\newcommand{\PGL}{\mathord{\mathrm{PGL}}}

\newcommand{\Aut}{\operatorname{\mathrm {Aut}}\nolimits}
\newcommand{\Gal}{\operatorname{\mathrm {Gal}}\nolimits}

\newcommand{\pr}{\mathord{\mathrm {pr}}}
\newcommand{\Sing}{\operatorname{\mathrm {Sing}}\nolimits}

\newcommand{\rmand}{\textrm{and}}

\newcommand{\quand}{\quad\rmand\quad}

\newcommand{\mystruthd}[2]{\phantom{\hbox{\vrule  height #1 depth #2}}}

\numberwithin{equation}{section}
\numberwithin{table}{section}
\numberwithin{figure}{section}
\renewcommand{\qed}{\hfill {$\Box$}}

\newcommand{\PGU}{\mathord{\mathrm {PGU}}}
\newcommand{\GU}{\mathord{\mathrm {GU}}}
\newcommand{\QT}{\mathord{QT}}

\newcommand{\QB}{\mathord{\mathcal{Q}}}

\newcommand{\vect}[1]{\mbox{\boldmath $#1$}}
\newcommand{\va}{\vect{a}}
\newcommand{\vp}{\vect{p}}
\newcommand{\vx}{\vect{x}}
\newcommand{\vy}{\vect{y}}

\newcommand{\smallva}{\mbox{\boldmath\scriptsize $a$}}

\newcommand{\transpose}[1]{\hskip 1.2pt {}^t\hskip -.8pt{#1}}

\newcommand{\intf}[3]{(#2, #3)_{#1}}
\newcommand{\intfNS}[2]{(#1, #2)_{\NS}}
\newcommand{\intfphantom}[1]{\intf{#1}{\phantom{\cdot\,}}{\phantom{\cdot}}}

\newcommand{\Pol}{\mathord{\mathcal{P}}}
\newcommand{\theh}{h_F}
\newcommand{\theNS}{M_{\NS}}
\newcommand{\Stab}{\mathord{\mathrm{Stab}}}

\newcommand{\RT}{\mathord{\mathrm{RT}}}

\newcommand{\isoms}{\mathord{\mathrm{isom}}}
\newcommand{\aut}{\mathord{\mathrm{aut}}}

\newcommand{\Id}{\mathord{\mathrm{Id}}}

\newcommand{\orb}{\mathord{o}}

\newcommand{\algcloF}{\overline{\F}}

\newcommand{\srt}{\sqrt{2}}

\newcommand{\projmodel}{\EEE}
\newcommand{\Ball}{\BBB}

\newcommand{\Exc}{\mathord{\mathrm{Exc}}}
\newcommand{\Lin}{\mathord{\mathrm{Lin}}}

\newcounter{algostepnumb}
\newcommand{\algostep}{\stepcounter{algostepnumb}\subsection{Step \thealgostepnumb} }

\newcommand{\hhline}[4]{$\ell\sb{#1}:=\ell\sp{#2}(#3)$}
\newlength{\NSbasisleng}
\newcommand{\PPMM}[9]{
\begin{eqnarray*}
\hskip -.9cm &&
\hbox{$\projmodel_{#1}=\overline{\projmodel}_{#2}$:}\;\;
\hbox{$\mathrm{RT}=#3$:}\;\;
\hbox{$|\mathrm{aut}|= #4$:}\;\; 
\hbox  {$\text{\rm N}= #7$:} \;\;  {\rm h}=#9:\\
\hskip -.9cm &&\fbox{\parbox{12.5cm}{$#6$}}
\end{eqnarray*}
\vskip -.35cm
}
\begin{document}

\title[Projective models of a supersingular $K3$ surface]%
{Projective models of the supersingular $K3$ surface with Artin invariant $1$
in characteristic $5$}

\author{Ichiro Shimada}
\address{
Department of Mathematics, 
Graduate School of Science, 
Hiroshima University,
1-3-1 Kagamiyama, 
Higashi-Hiroshima, 
739-8526 JAPAN
}
\email{shimada@math.sci.hiroshima-u.ac.jp
}

\thanks{Partially supported by
 JSPS Grants-in-Aid for Scientific Research (B) No.20340002 
}

\subjclass[2000]{14J28, 14G17}


\begin{abstract}
Let $X$ be a supersingular $K3$ surface in characteristic $5$ with Artin invariant $1$.
Then $X$ has a polarization   that realizes $X$
as the Fermat sextic double plane.
We present a list of polarizations 
of $X$ with degree $2$ whose intersection number with this Fermat sextic polarization  
is less than or equal to $5$,
and give the defining equations of the corresponding projective models.
We also present  a method to describe birational morphisms between  
these projective models explicitly.
As a by-product,
a non-projective automorphism of the Fermat sextic double plane is obtained.
\end{abstract}

\maketitle

%
%

\section{Introduction}
Let $Y$ be a supersingular $K3$ surface 
defined over an algebraically closed field of characteristic $p>0$.
Artin~\cite{MR0371899} showed that 
the discriminant of the  N\'eron-Severi lattice $\NS(Y)$ 
is written as $-p^{2\sigma}$,
where $\sigma$ is a positive integer  $\le 10$.
This integer $\sigma$ in called the \emph{Artin invariant} of $Y$.
It is proved in~\cite{MR563467, MR717616,  MR633161} that,
for  each prime $p$,
a supersingular $K3$ surface with Artin invariant $1$ 
in characteristic $p$ exists 
and is unique up to isomorphisms.
Recently, 
many detailed studies of
supersingular $K3$ surfaces with Artin invariant $1$ in small characteristics 
have appeared~(see~\cite{MR1935564K, ElkiesSchuett, MR2862188,
MR2211150, MR2987663, KondoShimada,   Sengupta}).
\par
\medskip
The purpose of this paper is to investigate  projective models 
of degree $2$ of the supersingular $K3$ surface $X$ with Artin invariant $1$
in characteristic $5$.
It is well-known that the Fermat  sextic  double plane
in characteristic $5$ is isomorphic to $X$.
This projective model enables us to calculate the defining ideals
of curves on $X$ whose classes generate $\NS(X)$.
Using this data,
we obtain many other projective models of  degree $2$,
present their explicit defining equations, 
and describe birational morphisms between them.
\par
\medskip  
Our method is  computational,
and can be easily adapted to 
other $K3$ surfaces in any characteristic.
In particular, we expect many geometric applications of  the algorithm in Section~\ref{subsec:polarizations}
that determines whether a given vector $v$ with $v^2>0$ in the N\'eron-Severi lattice of a $K3$ surface is nef or not. 
In fact, combining the algorithms developed in this paper  
with the Borcherds-Kondo method~\cite{MR913200, MR1654763, MR1618132}, 
we have succeeded in 
obtaining a set of generators of the automorphism group
of the supersingular $K3$ surface in characteristic $3$ with Artin invariant $1$ in~\cite{KondoShimada}.

\par
\medskip
We fix terminologies and explain our motivation.
Let $Y$ be a $K3$ surface defined over an algebraically closed field 
of arbitrary  characteristic.
Let  $\intfphantom{\NS}$ denote the intersection form 
of the N\'eron-Severi lattice $\NS(Y)$ of $Y$.
For  $v\in \NS(Y)$,
we denote by $\LLL_v\to X$ the corresponding line bundle.
Let $d$ be an even positive integer.
We say that a vector $h\in \NS(Y)$ is a \emph{polarization of degree $d$}
if $\intfNS{h}{h}$ is equal to $d$ and 
the complete linear system $|\LLL_h|$ is non-empty and has no fixed-components.
Let  $h$ be a polarization of degree $d$.
Then $|\LLL_h|$ is base-point free by Corollary 3.2 of~\cite{MR0364263},
and hence defines a morphism $\Phi_h$ from $Y$ to a projective space
of dimension $1+d/2$.
We denote by
$$
Y\maprightsp{\phi_h} Y_h\maprightsp{\psi_h} \P^{1+d/2}
$$
the Stein factorization of $\Phi_h$.
By~\cite{MR0146182, MR0199191}, 
the  normal surface $Y_h$ has  only rational double points as its singularities,
and $\phi_h$ is a contraction of an $ADE$-configuration of smooth rational curves.
We say that $\psi_h : Y_h\to \P^{1+d/2}$ is 
the \emph{projective model} of $Y$ corresponding to  $h$.
We put
$$
\Pol_d(Y):=\set{h \in \NS(Y)}{\textrm{$h$ is a polarization of degree $d$}}.
$$
The automorphism group $\Aut(Y)$ of $Y$ acts on  $\Pol_d (Y)$.
For $h, h\sprime\in \Pol_d  (Y)$, we say that
$h$ and $h\sprime$ are \emph{projectively equivalent} and write
$h\sim h\sprime$ if there exist an isomorphism $Y_h\isom Y_{h\sprime}$ and 
a linear automorphism $\P^{1+d/2}\isom \P^{1+d/2}$ that make the following diagram commutative:
\begin{equation}\label{eq:diagY}
\renewcommand{\arraystretch}{1.3}
\begin{array}{ccc}
 Y_h & \maprightsp{\psi_h} & \P^{1+d/2}\phantom{.} \\
\downarrow\hskip -2pt\wr & & \downarrow\hskip -2pt\wr \\
 Y_{h\sprime} & \maprightsp{\psi_{h\sprime}} & \P^{1+d/2}.
\end{array}
\end{equation}
It is obvious that the equivalence classes of $\sim$ in $\PPP_d(Y)$ are just 
the $\Aut(Y)$-orbits.
For $h\in \PPP_d(Y)$, 
the stabilizer subgroup  $\Aut(Y, h)$  of $h$ in  $\Aut(Y)$
 is the projective automorphism group of 
the projective model $\psi_h: Y_h\to \P^{1+d/2}$.
It is usually easy to determine $\Aut(Y, h)$.
Hence  it is important to study the equivalence classes 
of $\sim$ for the study of $\Aut(Y)$.
Moreover, 
to obtain an element of $\Aut(Y)$  \emph{not} contained in $\Aut(Y, h)$,
we need to write the isomorphism $Y_h\isom Y_{h\sprime}$ 
in~\eqref{eq:diagY} explicitly.
\par
\medskip
We concentrate upon the supersingular $K3$ surface $X$ with Artin invariant $1$
in characteristic $5$, and its projective models $\psi_h : X_h \to \P^2$ of degree $2$.
It is well-known that $X$ has a projective model $\psi_F: X_F\to \P^2$ of degree $2$,
where $X_F$ is defined by 
\begin{equation}\label{eq:XF}
X_F :=\{w^2=x^6+y^6+z^6\}\subset \P(3,1,1,1)
\end{equation}
in  the weighted projective space $\P(3,1,1,1)$,  
and the double covering $\psi_F$ is given by $[w:x:y:z]\mapsto [x:y:z]$, which 
is branching along the Fermat sextic curve
$$
B_F : x^6+y^6+z^6=0.
$$
We denote by  $\theh\in \NS(X)$  a polarization of the projective model $\psi_F: X_F\to \P^2$,
and by
$$
\Phi_F\;\;:\;\; X\maprightsp{\phi_F} X_F\maprightsp{\psi_F} \P^{2}
$$
the Stein factorization of the morphism given by $|\LLL_{\theh}|$.
Note that the factor $\phi_F: X\to X_F$ of $\Phi_F$ is an isomorphism.
The  group $\Aut(X, \theh)$ is 
an extension of the projective automorphism group $\PGU_3(\F_{25})$ 
of $B_F\subset \P^2$ by $\Gal (X_F/\P^2)\cong \Z/2\Z$.
In particular, 
the order of  $\Aut(X, \theh)$ is  $756,000$.
Using this projective model $\psi_F: X_F\to \P^2$,
we obtain a set of generators of $\NS(X)$~(see Section~\ref{sec:NSofX}). 
It turns out that $\NS(X)$ is generated by the numerical equivalence classes of 
curves on $X_F$ defined over $\F_{25}$.
In particular, 
every projective model of $X$ is projectively equivalent to a
projective model defined over $\F_{25}$
(see~\cite{MR2890513}).
Moreover, the Frobenius action of $\Gal(\F_{25}/\F_{5})$ on $X_F$ induces an 
action of $\Gal(\F_{25}/\F_{5})$ on $\NS(X)$,
which we denote by $v\mapsto \bar{v}$.
It is easy to see that
 $\Gal(\F_{25}/\F_5)$ acts on the set of $\Aut(X)$-orbits in $\Pol_d(X)$.
\par
\medskip
For each positive integer $r$, we consider the  subset
$$
\Ball_r:=\set{v\in \NS(X)}{\intfNS{v}{\theh}\le r}
$$
of $\NS(X)$, which can be regarded as a neighborhood of $h_F$ in $\NS(X)$.
Since $\overline{\theh}=\theh$,
 $\Aut(X, \theh)$ and $\Gal(\F_{25}/\F_{5})$ act on  $\Pol_2(X)\cap \Ball_r$.
 Our main result is the following:
%
%
%
\begin{theorem}\label{thm:projmodels}
The set $\Pol_2(X)\cap \Ball_5$ consists of $146,945,851$ vectors,
and they are decomposed into the equivalence classes $\projmodel_0, \dots, \projmodel_{64}$
under the relation $\sim$.
The details of these equivalence classes are described in  Section~\ref{sec:list}.
\end{theorem}
We explain the items of the table in Section~\ref{sec:list}.
For $h\in \PPP_2(X)$,
let $B_h$ denote the branch curve of the double covering $\psi_h: X_h\to \P^2$.

\begin{itemize}
\item
$\projmodel_i=\overline{\projmodel}_j$ means that $\projmodel_i$
is equal to the image of $\projmodel_j$ under the action of $\Gal(\F_{25}/\F_{5})$ 
defined above.
In particular, $\projmodel_i=\overline{\projmodel}_i$ means that $\projmodel_i$ is self-conjugate,
while $\projmodel_i=\overline{\projmodel}_{i+1}$ means that $\projmodel_i$ is \emph{not} self-conjugate,
that the  items $\mathrm{RT}$,  $|\mathrm{aut}|$
and  $\mathrm{N}$  explained
below  are the same for $\projmodel_i$ and $\projmodel_{i+1}$, and   that
the defining equation of  $B_h$ 
 for $\projmodel_{i+1}$ is obtained from that for $\projmodel_{i}$
 by raising the coefficients to $5$th powers; that is, 
by changing the sign of  $\sqrt{2}$.
\item 
$\RT$ denotes the $ADE$-type of the singular points of 
$B_h$.
\item
$|\aut|$  denotes the order of the projective automorphism group of 
the plane curve $B_h\subset \P^2$.
Hence the order of $\Aut(X, h)$ is equal to $2\,|\mathrm{aut}|$.
\item
$\mathrm{N}$ is the total number of the vectors in $\projmodel_i\subset \Pol_2(X)\cap \Ball_5$.
\item
$\mathrm{h}$ is a sample element of $\projmodel_i$ written in
a row vector  with respect to the basis of $\NS(X)$ given in Section~\ref{sec:NSofX}.
\item 
An affine defining equation of  $B_h$ 
with coefficients in $\F_{25}$ is given in the framed box.
\end{itemize}
Each of the $65$ projective models in Theorem~\ref{thm:projmodels}
exhibits  interesting properties that are peculiar to  characteristic $5$.
One of these properties is the existence of \emph{splitting lines}.
A \emph{$(-2)$-curve} on $X$ is a smooth rational curve on $X$.
Let $h$ be a polarization of degree $2$ on $X$.
We say that a $(-2)$-curve $C$ on $X$ is \emph{$h$-exceptional} if $C$ is mapped to a point
by $\Phi_h : X\to \P^2$, 
while
$C$ is said to be an \emph{$h$-line} if $\Phi_h$ maps $C$ to a line on $\P^2$
isomorphically.
A line $l$ on $\P^2$ is said to be \emph{$h$-splitting} 
if $l$ is the image of an $h$-line by $\Phi_h$.
In other words,
a line $l\subset \P^2$ is $h$-splitting 
if and only if either $l$ is an irreducible component of  $B_h$,
or $l\not\subset B_h$ and the intersection multiplicity at each point of $l\cap B_h$ is even.
We observe the following:
\begin{proposition}\label{prop:splgen}
For each $h\in \Pol_2(X)\cap \Ball_5$,
the lattice $\NS(X)$ is generated by the classes of $h$-exceptional curves and $h$-lines.
\end{proposition}
In fact, we establish a method to write the birational morphism
$\phi_h: X\to X_h$ explicitly as a list of rational functions on $X\cong  X_F$
for any $h\in \Pol_2(X)$.
Applying this method to a polarization $h\in \projmodel_0$
with $\intfNS{\theh}{h}=4$, we obtain the following:
\begin{example}\label{example:nonprojautXF}
There exists 
an automorphism $g$ of $X_F$ of order $2$ such that 
$\intfNS{\theh}{g^*\theh}=4$.
Moreover, we can write $g$  in  a form 
$$
(w,x,y)\mapsto [\omega(w,x,y):\xi_0(w,x,y):\xi_1(w,x,y):\xi_2(w,x,y) ],
$$
where $(w,x,y)$ are the affine coordinates of $\P(3,1,1,1)$ with $z=1$ in~\eqref{eq:XF},
and  $\omega, \xi_0, \xi_1, \xi_2$ are 
polynomials with coefficients in $\F_{25}$.
See~\cite{shimadaWWW} for the explicit presentation of these polynomials.
\end{example}
The study of  singularities of sextic double plane models of complex $K3$ surfaces
using lattice theory and computer-aided calculation was initiated 
by Urabe~\cite{MR1000608} and Yang~\cite{MR1387816}.
The idea of $h$-splitting lines was used in~\cite{MR2745755}
for the classification of Zariski pairs of simple sextic curves.
On the other hand, in~\cite{MR2282430, MR2129248, MR2036331}, 
sextic  double plane models of supersingular  $K3$ surfaces
were studied by lattice theory.
A shortcoming of the method in these works 
is that it gives only combinatorial data of the singularities of the projective models, and 
 does not  yield their defining equations explicitly.
\par
\medskip
 The new devices  in this article are the following:
 (i) Using the ample class $\theh\in \NS(X)$,
 we can determine whether a given vector $v\in \NS(X)$ is a polarization  or not.
 (ii)  The fact that the classes of $\theh$-lines span $\NS(X)$
enables us to calculate the equation of $X_h$ explicitly
and algorithmically.
(iii) To deal with the large number of polarizations,
we decompose them into $\Aut(X, \theh)$-orbits and calculate the projective model
only for a representative polarization of each orbit.
\par
\medskip
This paper is organized as follows.
In Section~\ref{sec:NSofX},
we give a set of $\theh$-lines whose classes form a basis
of $\NS(X)$.
In Section~\ref{sec:algorithm},
we present  algorithms 
that can be applied to lattices in general.
In Section~\ref{sec:geometricappli},
we apply them to $\NS(X)$ and describe algorithms
to calculate geometric data of $X$.
In Section~\ref{sec:morphism},
we explain how to calculate the morphisms
$\phi_h: X\to X_h$ and $\psi_h: X_h\to\P^2$
for a given polarization $h\in \Pol_2(X)$.
In Sections~\ref{sec:proof1} and~\ref{sec:proof2},
the computation we carried out to prove 
Theorem~\ref{thm:projmodels} and Example~\ref{example:nonprojautXF}
are explained.
Section~\ref{sec:list} is for the list of projective models.
\par
\medskip
{\bf Notation.}
(1)
A \emph{lattice} is a free $\Z$-module $L$ of finite rank 
with a non-degenerate symmetric bilinear form $\intfphantom{L}: L\times L\to \Z$.
\par
(2) 
The numerical equivalence class of a divisor $D$ on $X$ is denoted by $[D]\in \NS(X)$.
The intersection number of divisors $D$ and $D\sprime$ is written as 
$\intfNS{D}{D\sprime}$. 
\section{The N\'eron-Severi lattice of $X$}\label{sec:NSofX}
Recall that $B_F\subset \P^2$ is the Fermat curve of degree $6$
in characteristic $5$,
which is the branch curve of the projective model $\psi_F: X_F\to\P^2$
corresponding to the polarization $\theh\in \NS(X)$ of degree $2$.
We denote by $B_F(\F_{25})$ the set of $\F_{25}$-rational points of $B_F$.
It is known  that $|B_F(\F_{25})|=126$.
\par
\medskip
Let $l$ be a line on $\P^2$ tangent to $B_F$.
Since $B_F$ is the \emph{Hermitian curve} over $\F_{25}$,
either one of the following holds (see~\cite{MR0213949} or Chapter 23 of~\cite{MR1363259}):
\begin{itemize}
\item[(1)]  $l$ is tangent to $B_F$ at a point $[a:b:c]\notin B_F(\F_{25})$
 with intersection multiplicity $5$,
and intersects $B_F$  at the point $[a^{25}: b^{25}: c^{25}]$
transversely.
\item[(2)] $l$ is tangent to $B_F$ at $P\in B_F(\F_{25})$
with  intersection multiplicity $6$.
\end{itemize}
In the case (2), 
the inverse image of $l$ by the double covering $\Phi_F: X\to \P^2$
decomposes into two $\theh$-lines 
$\ell^+(P)$ and $\ell^-(P)$ such that
$$
\intfNS{\ell^+(P)}{ \ell^-(P)}=3.
$$
All $\theh$-lines  on $X$ are obtained as $\ell^{\pm }(P)$ with $P\in B_F(\F_{25})$.
In particular, 
the number of $\theh$-lines on $X$ is $252$.
We put 
$$
P_0:=[0:1:1+\srt]\in B_F(\F_{25}) \quand \ell^+(P_0):=\{x^3-w=0, y+(1-\srt)z=0\}.
$$
For $P\in B_F(\F_{25})\setminus\{P_0\}$,
we choose the sign of $\ell^\pm(P)$ in such a way that
$$
\intfNS{\ell^+(P)}{\ell^+(P_0)}=1\qquad
(\textrm{and hence $\intfNS{\ell^-(P)}{ \ell^+(P_0)}=0$}).
$$
From among these $\theh$-lines,
we choose the $22$ curves $\ell_1, \dots, \ell_{22}$ in Table~\ref{table:22hlines}.
\begin{table}
\begin{tabular}{p{\NSbasisleng}p{\NSbasisleng}}
\hhline{1}{+}{[0:1:1+\sqrt {2}]
}{ \left\{ y+4\,z\sqrt {2}+z=0,{x}^{3}+4\,w=0 \right\} 
} 
 & 
\hhline{2}{-}{[0:1:1+\sqrt {2}]
}{ \left\{ y+4\,z\sqrt {2}+z=0,w+{x}^{3}=0 \right\} 
} 
\\ 
\hhline{3}{+}{[0:1:1+4\,\sqrt {2}]
}{ \left\{ {x}^{3}+4\,w=0,y+z\sqrt {2}+z=0 \right\} 
} 
 & 
\hhline{4}{+}{[0:1:2]
}{ \left\{ {x}^{3}+4\,w=0,y+2\,z=0 \right\} 
} 
\\ 
\hhline{5}{+}{[0:1:3]
}{ \left\{ {x}^{3}+4\,w=0,y+3\,z=0 \right\} 
} 
 & 
\hhline{6}{+}{[0:1:4+\sqrt {2}]
}{ \left\{ {x}^{3}+4\,w=0,y+4\,z+4\,z\sqrt {2}=0 \right\} 
} 
\\ 
\hhline{7}{+}{[1:0:1+\sqrt {2}]
}{ \left\{ x+4\,z\sqrt {2}+z=0,{y}^{3}+4\,w=0 \right\} 
} 
 & 
\hhline{8}{+}{[1:0:1+4\,\sqrt {2}]
}{ \left\{ {y}^{3}+4\,w=0,x+z\sqrt {2}+z=0 \right\} 
} 
\\ 
\hhline{9}{+}{[1:0:2]
}{ \left\{ x+2\,z=0,w+{y}^{3}=0 \right\} 
} 
 & 
\hhline{10}{+}{[1:0:4+\sqrt {2}]
}{ \left\{ w+{y}^{3}=0,x+4\,z+4\,z\sqrt {2}=0 \right\} 
} 
\\ 
\hhline{11}{+}{[1:\sqrt {2}:1]
}{ \left\{ x+4\,\sqrt {2}y+z=0,{y}^{3}+2\,z\sqrt {2}{y}^{2}+3\,w+{z}^{2}
y+3\,{z}^{3}\sqrt {2}=0 \right\} 
} 
 & 
\hhline{12}{-}{[1:\sqrt {2}:2+2\,\sqrt {2}]
}{ \left\{ x+4\,\sqrt {2}y+3\,z\sqrt {2}+2\,z=0,3\,{z}^{3}\sqrt {2}+2\,{
z}^{2}y+2\,{z}^{2}\sqrt {2}y+3\,w+{y}^{3}+2\,z{y}^{2}+4\,z\sqrt {2}{y}
^{2}=0 \right\} 
} 
\\ 
\hhline{13}{-}{[1:\sqrt {2}:2+3\,\sqrt {2}]
}{ \left\{ 3\,{z}^{3}\sqrt {2}+2\,{z}^{2}y+3\,{z}^{2}\sqrt {2}y+3\,w+{y}
^{3}+3\,z{y}^{2}+4\,z\sqrt {2}{y}^{2}=0,x+4\,\sqrt {2}y+2\,z\sqrt {2}+
2\,z=0 \right\} 
} 
 & 
\hhline{14}{+}{[1:\sqrt {2}:3+2\,\sqrt {2}]
}{ \left\{ x+4\,\sqrt {2}y+3\,z+3\,z\sqrt {2}=0,2\,{z}^{3}\sqrt {2}+2\,{
z}^{2}y+3\,{z}^{2}\sqrt {2}y+2\,w+{y}^{3}+2\,z{y}^{2}+z\sqrt {2}{y}^{2
}=0 \right\} 
} 
\\ 
\hhline{15}{-}{[1:\sqrt {2}:3+3\,\sqrt {2}]
}{ \left\{ x+4\,\sqrt {2}y+2\,z\sqrt {2}+3\,z=0,2\,{z}^{3}\sqrt {2}+2\,{
z}^{2}y+2\,{z}^{2}\sqrt {2}y+2\,w+{y}^{3}+3\,z{y}^{2}+z\sqrt {2}{y}^{2
}=0 \right\} 
} 
 & 
\hhline{16}{+}{[1:2\,\sqrt {2}:2\,\sqrt {2}]
}{ \left\{ {y}^{3}+4\,{z}^{3}+2\,z{y}^{2}+4\,\sqrt {2}w+3\,{z}^{2}y=0,x+
3\,\sqrt {2}y+3\,z\sqrt {2}=0 \right\} 
} 
\\ 
\hhline{17}{+}{[1:2\,\sqrt {2}:3\,\sqrt {2}]
}{ \left\{ x+3\,\sqrt {2}y+2\,z\sqrt {2}=0,{y}^{3}+{z}^{3}+3\,z{y}^{2}+
\sqrt {2}w+3\,{z}^{2}y=0 \right\} 
} 
 & 
\hhline{18}{-}{[1:2\,\sqrt {2}:2+\sqrt {2}]
}{ \left\{ x+3\,\sqrt {2}y+2\,z+4\,z\sqrt {2}=0,{z}^{3}+{z}^{2}y+{z}^{2}
\sqrt {2}y+\sqrt {2}w+{y}^{3}+z{y}^{2}+4\,z\sqrt {2}{y}^{2}=0
 \right\} 
} 
\\ 
\hhline{19}{+}{[1:2\,\sqrt {2}:2+4\,\sqrt {2}]
}{ \left\{ x+3\,\sqrt {2}y+2\,z+z\sqrt {2}=0,4\,{z}^{3}+{z}^{2}y+4\,{z}^
{2}\sqrt {2}y+4\,\sqrt {2}w+{y}^{3}+4\,z{y}^{2}+4\,z\sqrt {2}{y}^{2}=0
 \right\} 
} 
 & 
\hhline{20}{+}{[1:2\,\sqrt {2}:3+\sqrt {2}]
}{ \left\{ x+3\,\sqrt {2}y+3\,z+4\,z\sqrt {2}=0,{z}^{3}+{z}^{2}y+4\,{z}^
{2}\sqrt {2}y+\sqrt {2}w+{y}^{3}+z{y}^{2}+z\sqrt {2}{y}^{2}=0
 \right\} 
} 
\\ 
\hhline{21}{+}{[1:1+\sqrt {2}:0]
}{ \left\{ x+y+4\,\sqrt {2}y=0,w+{z}^{3}=0 \right\} 
} 
 & 
\hhline{22}{+}{[1:1+3\,\sqrt {2}:1]
}{ \left\{ x+y+2\,\sqrt {2}y+z=0,2\,{z}^{3}\sqrt {2}+2\,{z}^{2}y+3\,{z}^
{2}\sqrt {2}y+3\,w+{y}^{3}+2\,z{y}^{2}+z\sqrt {2}{y}^{2}=0 \right\} 
} 
\\ 
\end{tabular}
%
%
\caption{Basis of $\NS(X)$}\label{table:22hlines}
\end{table}
Then their intersection matrix $M_{\NS}$ is calculated
to have  $\det \theNS=-25$.
(See~\cite{shimadaWWW} for the explicit presentation of $M_{\NS}$.)
Hence the classes of $\ell_1, \dots, \ell_{22}$
form a $\Z$-basis of $\NS(X)$.
We fix this basis throughout the paper.
Each element of $\NS(X)$
is written as a \emph{row} vector 
with respect to this basis.
In particular,
the orthogonal group  $\mathrm{O}(\NS(X))$ of the lattice $\NS(X)$ acts on $\NS(X)$ from the right.
Since $\theh=[\ell^+(P)]+[\ell^-(P)]$ for any  $P\in B_F(\F_{25})$,
we have
\begin{equation}\label{eq:theh}
\theh=[1,1,0,0,0,0,0,0,0,0,0,0,0,0,0,0, 0,0,0,0,0,0].
\end{equation}
We calculate the vector representations of the classes  of 
all $\theh$-lines.
\begin{example}
The class of the $\theh$-line
$\ell^-([1:4+4\,\sqrt {2}:0])$ 
is 
$$
[-4,-6,3,1,1,2,1,-1,2,1,1,4,1,0,-3,0,2,-1,3,-1,-2,-3].
$$
\end{example}
From the action of $\PGU_3(\F_{25})$ on the set $B_F(\F_{25})$,
we can calculate 
the action of  $\Aut(X, \theh)$
on the set of $\theh$-lines.
Using this permutation representation,
we can write explicitly
the linear representation
\begin{equation}\label{eq:rep}
\Aut (X, \theh)\to\set{T\in \GL_{22}(\Z)}{T M_{\NS} \transpose{\hskip .8pt T} =M_{\NS}}
\cong \mathrm{O}(\NS(X)).
\end{equation}
This representation is faithful (see Proposition~3 in Section~8 of~\cite{MR633161}).
\begin{remark}\label{rem:storerep}
The representation~\eqref{eq:rep} is encoded
as follows.
We number the $\theh$-lines as
$\ell_1, \dots, \ell_{22}, \ell_{23}, \dots, \ell_{252}$
once and for all.
Then each $\gamma\in \Aut(X, \theh)$ is labelled by a list of 
$22$ integers $[n_\gamma(1), \dots, n_\gamma(22)]$ in such a way that 
the image $\ell_i^\gamma$ of $\ell_i$ by $\gamma$
is equal to $\ell_{n_\gamma(i)}$ for $i=1, \dots, 22$.
Then the action of $\gamma$ on $\NS(X)$ is given by 
$v\mapsto v T_{\gamma}$, where $T_{\gamma}$ is the $22\times 22$ matrix 
whose $i$th row vector is $[\ell_{n_\gamma(i)}]$.
\end{remark}
The Galois group $\Gal (\F_{25}/\F_{5})$ also acts on the set of $\theh$-lines
by the Frobenius action on $X_F$.
We denote by  $\Gamma_{\NS}$ the matrix 
that represents
this Frobenius conjugate action $v\mapsto \bar v=v\Gamma_{\NS}$ on $\NS(X)$
with respect to the basis $\ell_1, \dots, \ell_{22}$.
See~\cite{shimadaWWW} for the explicit presentation of $\Gamma_{\NS}$.
\section{Algorithms for lattices}\label{sec:algorithm}
\subsection{An algorithm for a positive quadratic triple}\label{subsec:QT}
By a \emph{quadratic triple} of $n$-variables, we
mean a triple $[Q, L, c]$,
where $Q$ is an $n\times n$  symmetric matrix with entries in $\Q$,
$L$ is a column vector of length $n$ with entries in $\Q$, and $c$ is a rational number.
An element of $\R^n$ is written as a row vector $\vx=[x_1, \dots, x_n]$.
The  \emph{inhomogeneous quadratic function}
$q_{\QT} : \Q^n\to \Q$
associated with a  quadratic triple $\QT=[Q, L, c]$
is defined by 
$$
q_{\QT} (\vx):=\vx\, Q \,\transpose{\vx} + 2\, \vx\, L + c.
$$
We say that  $\QT=[Q, L, c]$ 
and   $q_{\QT}$ are \emph{positive}
or~\emph{negative}
according to whether the symmetric matrix  $Q$ is positive-definite or negative-definite.
\par
\medskip
Let $\QT=[Q, L, c]$ be a positive quadratic triple
of $n$-variables.  
In this section, we describe an algorithm to calculate the finite set
$$
E(\QT):=\set{\vx\in \Z^n}{q_{\QT}(\vx)\le 0}.
$$
Suppose that $\QT=[Q, L, c]$ is written as follows:
$$
\mystruthd{35pt}{30pt}
Q= \left[ 
\begin {array}{ccc|c} 
&&&\\
&Q\sprime&&\vp\sprime\\
&&&\\
\hline
&\transpose{\vp\sprime}\mystruthd{12pt}{0pt}&&r\sprime
\end {array} 
\right]
=\left[ 
\begin {array}{c|ccc} r\spprime&&\transpose{\vp\spprime}\mystruthd{0pt}{5pt}&\\
\hline
&&&\\
\vp\spprime&&Q\spprime&\\
&&&
\end {array} 
\right],
\quad
L=\left[
\begin {array}{c}
\\
L\sprime\\
\\
\hline
m\sprime \mystruthd{12pt}{0pt}
\end{array}
\right]
=\left[
\begin {array}{c}
m\spprime \mystruthd{0pt}{5pt}\\
\hline
\\
L\spprime\\
\\
\end{array}
\right],
$$
where $Q\sprime$ and $Q\spprime$ are square matrices of size $n-1$,
$\vp\sprime$, $\vp\spprime$, $L\sprime$ and  $L\spprime$ are column vectors of length $n-1$,
and $r\sprime$, $r\spprime$, $m\sprime$ and $m\spprime$ are rational numbers.
Note that, since $Q$ is positive-definite, we have $r\sprime>0$ and $r\spprime>0$.
We define a positive quadratic triple $\pr (\QT)$ of $(n-1)$-variables by
$$
\pr (\QT):=\left[\;
Q\sprime-\frac{1}{r\sprime}(\vp\sprime\transpose{\vp\sprime}),\;
L\sprime-\frac{m\sprime}{r\sprime}\vp\sprime,\;
c-\frac{m\sp{\prime 2}}{r\sprime}
\;\right].
$$
Then, for each $t\in \R$, the compact subset
$\shortset{\vx\in \R^n}{q_{\QT} (\vx)\le t}$
of $\R^n$ is mapped by the projection
$[x_1, \dots, x_n]\mapsto [x_1, \dots, x_{n-1}]$
to the compact subset
$$
\set{\vy\in \R^{n-1}}{q_{\pr(\QT)} (\vy)\le t}
$$ 
of $\R^{n-1}$.
For $a\in \Q$, we define a  positive quadratic  triple $\iota^*(a, \QT)$ of $(n-1)$-variables by
$$
\iota^*(a, \QT):=
[\;Q\spprime,\; a\,\vp\spprime+L\spprime, \; a^2\, r\spprime +2\, a\, m\spprime +c\;],
$$
and, for $\va=[a_1, \dots, a_m]\in \Q^m$ with $m<n$, 
we define a positive quadratic triple $\iota^*(\va, \QT)$ of $(n-m)$-variables by
$$
\QT^{0}:=\QT, \quad \QT^{\nu+1}:=\iota^*(a_{\nu+1}, \QT^{\nu}) 
\;\; (\nu=0, \dots, m-1), \quad \iota^*(\va, \QT):=\QT^{m}.
$$
Then the positive inhomogeneous quadratic function 
$q_{\iota^*({\smallva}, \QT)}: \Q^{n-m}\to \Q$
is equal to the composite $q_{\QT}\circ \iota_{{\smallva}}$,
where $\iota_{{\smallva}}$ is  the inclusion $\Q^{n-m}\inj \Q^n$
given by
$$
[y_1, \dots, y_{n-m}]\mapsto [a_1, \dots, a_m, y_1, \dots, y_{n-m}].
$$
Suppose that  $\va=[a_1, \dots, a_{n-1}]\in E(\pr(\QT))$
is given. Then the positive quadratic triple $\iota^*(\va, \QT)$
is of  \emph{one} variable, and 
the fiber of the projection $E(\QT)\to E(\pr(\QT))$
over $\va$ is equal to
$$
\set{[a_1, \dots, a_{n-1}, b]}{b \in E(\iota^*(\va, \QT))}.
$$
Since 
 $E(\iota^*(\va, \QT))$ is easily calculated,
we can obtain $E(\QT)$ if we know $E(\pr(\QT))$.
Using this idea iteratively,
we carry out  the following computation.
\par
\medskip
Starting from  the given positive quadratic triple 
$\QT_n^0:=\QT$ of $n$-variables,
we compute positive quadratic triples 
$\QT^0_\mu$ of $\mu$-variables by
$$
\QT_{\mu}^0:=\pr (\QT^0_{\mu+1})
\qquad (\mu=n-1, \dots, 1).
$$
We prepare an empty set $E:=\{\;\}$.
We then write a program $\QB (\nu, \va)$ that takes 
an integer $\nu\le n+1$ and
a vector $\va=[a_1, \dots, a_{\nu-1}]\in \Z^{\nu-1}$ as input,
and carries out the  task below.
Note that, when $\QB(\nu, \va)$ starts with $\nu>1$,
$\va$ is an element of $E(\QT_{\nu-1}^0)$,
and for $\mu>\nu-1$,
$\QT_{\mu}^{\nu-1}$ is the positive quadratic triple $\iota^*(\va, \QT_{\mu}^0)$
of $(\mu-\nu+1)$-variables.
In particular, $\QT_{\nu}^{\nu-1}$ is of one variable.
\par
\medskip
The task of $\QB (\nu, \va)$:
\begin{itemize}
\item[(1)] If $\nu=n+1$, then $\QB (\nu, \va)$ appends $\va$ to the set $E$.
\item[(2)] If $\nu\le n$, then the program $\QB (\nu, \va)$ 
\begin{itemize}
\item[(2-i)]  calculates the set $E(\QT_{\nu}^{\nu-1})=\{b_1, \dots, b_N\}$, and
\item[(2-ii)] for each $b_i \in E(\QT_{\nu}^{\nu-1})$, 
\begin{itemize}
\item[(2-ii-a)] computes $\QT_\mu^\nu:=\iota^*(b_i, \QT_\mu^{\nu-1})$ for $\mu=\nu+1, \dots, n$, and
\item[(2-ii-b)] proceeds to execute $\QB(\nu+1, [a_1, \dots, a_{\nu-1}, b_i])$.
\end{itemize}
\end{itemize}
\end{itemize}
We execute $\QB(1, [\;])$.
Since each $E(\QT_{\nu}^{\nu-1})$ is finite,
this program certainly terminates.
When the whole computation halts,
the set $E$ is equal to $E(\QT)$.
\subsection{An application to hyperbolic lattices I}\label{subsec:applicationI}
Changing the sign, we can apply the algorithm 
above to \emph{negative} inhomogeneous 
quadratic functions.
\par
\medskip
Suppose that $N$ is a hyperbolic lattice of rank $n$, that is, the signature of $\intfphantom{N}$ is $(1, n-1)$.
Let $\shortset{[v_i,a_i]}{i=1, \dots, k}$ be 
a finite set of pairs of $v_i\in N$ and $a_i\in \Z$
such that $\intf{N}{v_i}{v_i}>0 $ for at least one $i$, 
and let $d$ be an integer.
We can calculate the set
\begin{equation}\label{eq:cutlattice}
\set{x\in N}{\intf{N}{x}{v_i}=a_i\;\;\textrm{for}\;\; i=1, \dots, k,
\;\;\textrm{and}\;\; \intf{N}{x}{x} = d}
\end{equation}
by the following method.
We put
$$
M:=\set{x\in N}{\intf{N}{x}{v_i}=a_i\;\;\textrm{for}\;\; i=1, \dots, k}.
$$
%
It is  easy to determine whether $M$ is empty or not.
Suppose that $M\ne \emptyset$.
By choosing a point $c\in M$ as an origin,
we can regard $M$ as a free $\Z$-module of finite rank.
By the assumption on  $v_i$,
the restriction of $\intfphantom{N}$ to $M\subset N$ defines 
a negative  inhomogeneous  quadratic function on $M$.
Therefore we can calculate the set~\eqref{eq:cutlattice}
by the algorithm in Section~\ref{subsec:QT}.
\subsection{An application to hyperbolic lattices II}\label{subsec:applicationII}
Let $N$ be as  in the previous subsection.
Suppose that we are given vectors $h, v\in N$
satisfying
\begin{equation}\label{eq:assumphv}
\intf{N}{h}{h}>0,
\quad
\intf{N}{v}{v}>0,
\quad
\intf{N}{h}{v}>0.
\end{equation}
We describe an algorithm
that calculates,
for a given integer $d$,
the set
\begin{equation}\label{eq:setS}
S:=\set{r\in N}{\intf{N}{r}{h}>0, \;\intf{N}{r}{v}<0, \;\intf{N}{r}{r}=d}.
\end{equation}
Consider the orthogonal direct-sum decomposition
$N\tensor \R=\gen{h}\oplus \gen{h}\sperp$.
We denote the second projection by
$\pr_2: N\tensor \R\to \gen{h}\sperp$,   
and put
$$
W:=\pr_2(N),
$$
which is a free $\Z$-module of rank $n-1$ such that $W\tensor \R= \gen{h}\sperp$.
Note that $W\subset N\tensor\Q$.
We denote by
$$
\intfphantom{W}: W\times W\to \Q
$$
the restriction of $\intfphantom{N}$ to $W$.
Suppose that
$x\in N\tensor \R$ satisfies $\intf{N}{h}{x}\ne 0$ and $\intf{N}{x}{x}> 0$.
Then the composite
\begin{equation}\label{eq:compositexh}
\gen{x}\sperp \inj N\tensor \R \maprightsp{\pr_2} \gen{h}\sperp
\end{equation}
is an isomorphism of $\R$-vector spaces.
Let $\varphi_x : \gen{h}\sperp \;\isom\; \gen{x}\sperp$
denote the inverse of the isomorphism~\eqref{eq:compositexh}, that is,
$$
\varphi_x (y)=y-\frac{\intf{N}{y}{x}}{\intf{N}{h}{x}} h\quad\textrm{for}\quad y\in\gen{h}\sperp.
$$
We then define 
$f_x : \gen{h}\sperp\to \R$
by
$$
f_x(y):=\intf{N}{\varphi_x(y)}{\varphi_x(y)}=
\intf{W}{y}{y}+\frac{\intf{N}{y}{x}^2}{\intf{N}{h}{x}^2}\,\intf{N}{h}{h}
\;\;\;\;\textrm{for}\;\;\;\; y\in\gen{h}\sperp=W\tensor \R.
$$
Since $\intf{N}{x}{x}>0$, the real quadratic form $\intfphantom{N}$ 
restricted to $\gen{x}\sperp$
is negative-definite, and hence so is $f_x$.
By the condition~\eqref{eq:assumphv},
we see that $f_{h+tv}$ is negative-definite on $W\tensor \R$ for 
any $t\in \R_{\ge 0}\cup\{\infty\}$.
(Here we understand that $f_{h+\infty v}=f_v$.)
\par
\medskip
For simplicity, we put
$$
c_h:=\intf{N}{h}{h}, 
\quad
c_v:=\intf{N}{h}{v},
\quad
v_W:=\pr_2(v)\in W.
$$
Let  $x\sprime$ be a vector in  $\gen{h}\sperp=W\tensor\R$.
Since $v-v_W\in \gen{h}$, 
we have  
\begin{equation}\label{eq:fhtv}
f_{h+tv}(x\sprime)
=\intf{W}{x\sprime}{x\sprime}+\frac{t^2 \intf{W}{x\sprime}{v_W}^2}{(c_h +t c_v)^2} c_h.
\end{equation}
By~\eqref{eq:assumphv},
we have $c_h/c_v>0$, and hence, 
for a fixed $x\sprime\in \gen{h}\sperp$,
$f_{h+tv}(x\sprime)$ is a non-decreasing function with respect to $t\in \R_{\ge 0}$
bounded from above by
$$
f_{h+\infty v}(x\sprime)=\intf{W}{x\sprime}{x\sprime}+\frac{\intf{W}{x\sprime}{v_W}^2}{c_v^2} c_h.
$$
Note that  $f_{h+\infty v}$  restricted to  $W\subset W\tensor \R$
is  $\Q$-valued,
and hence $f_{h+\infty v}$ is a  negative inhomogeneous quadratic function on $W\tensor\Q$.
Applying the algorithm in Section~\ref{subsec:QT} to 
$f_{h+\infty v}$, 
we can calculate the finite set
$$
S_W:=\set{r\sprime \in W}{f_{h+\infty v}(r\sprime)\ge d},
$$
where $d$ is the  integer given as input.
\par
\medskip
Suppose that
$r$ is an element of the set $S$ in~\eqref{eq:setS}.
We put
$$
t_r:=-\frac{\intf{N}{r}{h}}{\intf{N}{r}{v}}\;\;\in\;\;\R_{>0}.
$$
Then we have $r\in \gen{h+t_r v}\sperp$.
We put $r\sprime:=\pr_2(r)\in W$.
Since $\varphi_{h+t_r v}(r\sprime)=r$, we have 
$$
d=\intf{N}{r}{r}=f_{h+t_r v}(r\sprime)\le f_{h+\infty v}(r\sprime).
$$
Therefore $r\sprime\in S_W$ holds.
Let $\rho\in \Q$ be the rational number such that
$r=\rho h+r\sprime$.
Since $\intf{N}{r}{r}=d$,  $\intf{N}{r\sprime}{h}= 0$ and $\intf{N}{r}{h}> 0$,
we have  
\begin{equation}\label{eq:rho}
\rho=\frac{\intf{N}{r}{h}}{c_h}=\sqrt{\frac{d-\intf{W}{r\sprime}{r\sprime}}{c_h}}.
\end{equation}
The right-hand side  of~\eqref{eq:rho} can be calculated if we know $r\sprime\in W$.
\par
\medskip
Therefore we obtain $S$ from $S_W$ by the following method.
First we set $S=\{\phantom{a}\}$.
For each $r\sprime\in S_W$,
we put
$$
\rho\sprime:=\sqrt{\frac{d-\intf{W}{r\sprime}{r\sprime}}{c_h}}
\quand
r:=\rho\sprime h+r\sprime\in N\tensor \R.
$$
We then  determine whether $r$ is contained in $N$ or not.
(If $\rho\sprime\notin \Q$, then we obviously have $r\notin N$.)
If $r\in N$, $\intf{N}{r}{h}>0$  and   $\intf{N}{r}{v}<0$, 
we append $r$ to $S$.
When this calculation  is done for all $r\sprime\in S_W$,
the set $S$ is equal to the set~\eqref{eq:setS}.
\section{Geometric applications}\label{sec:geometricappli}
We apply the algorithms above to the hyperbolic lattice $\NS(X)$.
%
\subsection{Polarizations}\label{subsec:polarizations}
If  $v\in \NS(X)$ is a polarization,
then we necessarily have $\intfNS{v}{v}>0$ and $\intfNS{v}{\theh}>0$.
It is well-known that 
the nef cone of $X$ is bounded by the hyperplanes perpendicular to classes of $(-2)$-curves
(see Section 3 of~\cite{MR633161}, for example).
If $v$ with $\intfNS{v}{v}>0$ is nef, then Proposition 0.1 of~\cite{MR1260944} gives a criterion for 
$v$ to be a polarization.
Thus we obtain the following:
\begin{proposition}\label{prop:polarization}
Suppose that a vector  $v\in \NS(X)$ satisfies $\intfNS{v}{v}>0$ and $\intfNS{v}{\theh}>0$.
Consider  the sets
\begin{eqnarray*}
S_1 &:=& \set{r\in\NS(X)}{\intfNS{r}{r}=-2, \;\;\intfNS{r}{\theh}>0,  \;\;\intfNS{r}{v}<0} \quand\\
S_2 &:=& \set{e\in\NS(X)}{\intfNS{e}{e}=0,  \;\;\intfNS{e}{v}=1}.
\end{eqnarray*}
Then $v$ is  nef if and only if $S_1=\emptyset$.
If $v$ is nef,
then $v$ is a polarization if and only if $S_2=\emptyset$.
\end{proposition}
The sets $S_1$ and  $S_2$ can be calculated  by the algorithms 
in Sections~\ref{subsec:applicationII} and~\ref{subsec:applicationI},
respectively.
Hence 
Proposition~\ref{prop:polarization} enables us to determine whether 
a given vector $v\in \NS(X)$ is a polarization or not.
\subsection{$h$-Exceptional curves}\label{subsec:hexceptional}
Let $h\in \NS(X)$ be a polarization of arbitrary degree.
A $(-2)$-curve $C$ on $X$ is called \emph{$h$-exceptional}
if $\Phi_h$ contracts $C$.
The set $\Exc(h)\subset \NS(X)$ of the classes of $h$-exceptional curves is 
calculated by the following algorithm.
We calculate the finite set
$$
R:=\set{r\in \NS(X)}{\intfNS{r}{r}=-2,  \;\;\intfNS{r}{h}=0}
$$
by the algorithm in Section~\ref{subsec:applicationI}, and 
classify the elements of $R$ by the degree with 
respect to the ample class $\theh$ as follows:
$$
R[m]:=\set{r\in R}{\intfNS{r}{\theh}=m}\quad\textrm{and}
\quad
R^+:=\bigcup_{m>0} R[m].
$$
We say that $r\in R^+$ is \emph{indecomposable} if there are no vectors $r_1, \dots, r_k\in R^+$
with $k>1$ 
such that $r=r_1+\cdots+r_k$.
Since each $R[m]$ is finite,
we can determine whether a given vector $r\in R^+$ is indecomposable or not.
It is obvious that $r\in R^+$ is contained in $\Exc(h)$ if and only if 
$r$ is  indecomposable.
\subsection{$h$-Lines}\label{subsec:hlines}
Let $h\in \NS(X)$ be a polarization of arbitrary degree.
A $(-2)$-curve $C$ on $X$ is called an \emph{$h$-line}
if $\Phi_h$ maps  $C$ to a line isomorphically.
The set $\Lin(h)\subset \NS(X)$ of the classes of $h$-lines is
calculated by the following algorithm.
We calculate the finite sets
\begin{eqnarray*}
&&L:=\set{r\in \NS(X)}{\intfNS{r}{r}=-2,  \;\;\intfNS{r}{h}=1}, \\
&&L[m]:=\set{r\in L}{\intfNS{r}{\theh}=m},
\quad
L^+:=\bigcup_{m>0} L[m].
\end{eqnarray*}
It is obvious that $\Lin(h)\subset L^+$.
If $r\in L^+$,
then we see that $r$ is the class of an effective divisor $D$,
that exactly one irreducible component $D_0$ of $D$ is an $h$-line,
and that $D-D_0$ is a finite sum of $h$-exceptional curves.
Hence $r\in L^+$ is contained in $\Lin(h)$ if and only if 
there are no $r\sprime\in L[m\sprime]$ with $m\sprime< \intfNS{r}{\theh}$ and  
$r_1,  \dots, r_k\in\Exc(h)$ 
with $k\ge 1$ such that
$r=r\sprime+r_1+\dots+r_k$.
 Since each of $L[m\sprime]$ and $\Exc(h)$ are  finite,
 we can determine the subset $\Lin(h)\subset L^+$.
\section{Explicit defining equations}\label{sec:morphism}
We identify $X$ with $X_F$ by the isomorphism $\phi_{F} : X\isom X_F$,
so that, for a polarization $h\in \PPP_2(X)$, 
we consider  $\Phi_h: X\to \P^2$ and  $\phi_h: X\to X_h$ as morphisms from $X_F$.
In this section, we describe a method to write the morphisms 
$\Phi_h$ and $\phi_h$ as  lists of rational functions 
on $X_F$ over $\F_{25}$.
%
\subsection{The global sections of a line bundle}\label{subsec:linebundle}
Let $H_{\infty}\subset X_F$ denote  the hyperplane section
defined by $z=0$ in~\eqref{eq:XF}.
We use the affine coordinates $(w, x, y)$ of $\P(3,1,1,1)$ with $z=1$,
and put
$$
F:=w^2-x^6-y^6-1\;\;\in\;\;  \F_{25} [w,x,y].
$$
For any $g\in \F_{25} [w,x,y]$,
there exists a unique polynomial $\bar{g}^F$
of the form $w f + h$ with $f, h\in \F_{25} [x,y]$
such that
$$
g\equiv  \bar{g}^F\bmod(F)\;\;\;\textrm{in}\;\;\; \F_{25}[w,x,y].
$$
We call $\bar{g}^F$ the \emph{normal form} of $g$.
Let $m$ be an integer.
By identifying the line bundle $\LLL_{m\theh}\to X$ 
with the invertible sheaf $\OOO_{X_F}(m H_{\infty})$,
the vector space $\Gamma(X, \LLL_{m\theh})$ of the global sections of $\LLL_{m\theh}$ 
defined over $\F_{25}$ is naturally identified with
the vector subspace
$$
V_m:=\set{w f + h}{f, h\in \F_{25} [x,y], \;\;\deg f\le m-3,\;\; \deg h\le m}
$$
of $\F_{25} [w,x,y]$.
%
%
Recall that all $\theh$-lines are defined over $\F_{25}$, 
and that no $\theh$-lines are contained in $H_{\infty}$.
We have indexed the $\theh$-lines 
as $\ell_1, \dots, \ell_{252}$ in Remark~\ref{rem:storerep}.
For $j=1, \dots, 252$, we denote by
$$
I_{j}\subset \F_{25} [w,x,y]
$$
the inhomogeneous  ideal defining  
$\ell_j$ in $\P(3,1,1,1)$, and put
$$
I_{j}\spar{\nu}:=I_{j}^\nu + (F) \subset \F_{25} [w,x,y] \quad\textrm{for}\;\; \nu\in \Z_{>0}.
$$
\par
\medskip
We describe an algorithm that takes a vector $v\in \NS(X)$ as input,
and calculates the vector space $\Gamma(X, \LLL_{v})$ of the global sections
of the corresponding line bundle $\LLL_{v}\to X$ defined over $\F_{25}$.
Using the $\Z$-basis $[\ell_1], \dots, [\ell_{22}]$ of $\NS(X)$,
$v$ is uniquely written as 
$$
v=\sum_{i\in J^+} a_i [\ell_i]-\sum_{j\in J^-} b_j [\ell_j],  
$$
where $J^+$ and $J^-$ are disjoint subsets of $\{1,\dots, 22\}$,
and $a_i, b_j$ are positive integers.
Let $i\sprime$ be the index of the $\theh$-line $\ell_{i\sprime}$
that is the image of $\ell_i$ by 
the deck-transformation of  $X_F$ over $\P^2$.
Since $[\ell_i]+[\ell_{i\sprime}]=\theh$ for any $i$,
we have
\begin{equation*} 
v=d\sprime (v) \theh- \sum_{i\in J^+} a_i [\ell_{i\sprime}]-\sum_{j\in J^-} b_j [\ell_j],
\quad\textrm{where $d\sprime (v) :=\sum_{i\in J^+} a_i$.}
\end{equation*}
Thus we  have an expression 
\begin{equation}\label{eq:dv}
v=d(v) \theh- \sum_{j\in J} c_j [\ell_j],
\end{equation}
where $d(v)$ is a non-negative integer, $J$ is a subset of $\{1, \dots, 252\}$,
and $c_j$ are positive integers.
(Since there are linear relations among $[\ell_j]$,
this expression is  not unique.)
Then the vector space $\Gamma(X, \LLL_{v})$ is identified with
the space of global sections   of  
$\OOO_{X_F}(d(v)H_{\infty})$
that vanish along $\ell_j$ with order $c_j$ for each $j\in J$,
that is,
\begin{equation}\label{eq:intersection}
\Gamma(X, \LLL_{v})\;\cong\; V_{d(v)}\cap\bigcap_{j\in J}{I_{j}\spar{c_j}},
\end{equation}
where the intersections are taken in $\F_{25}[w,x,y]$.
From now on,
we regard $\Gamma(X, \LLL_{v})$ as a subspace of $V_{d(v)}$
by~\eqref{eq:intersection}.
%
%
The vector space $ V_{d(v)}$ has a basis 
$$
m_{\alpha}:=wM\;\;\textrm{or}\;\; N\qquad (\alpha=1, \dots, 2+d(v)^2),
$$
where $M$ and $N$ are the monomials of $x$ and $y$ with  $\deg M \le d(v)-3$ and  $\deg N \le d(v)$.
We calculate the  Gr\"obner basis  $G_j$ of the ideal $I_{j}\spar{c_j}\subset \F_{25}[w,x,y]$
for each $j\in J$.
(In the actual calculation,
we used the graded reverse lexicographic order ${\tt grevlex}(w,x,y)$. 
See Chapter~2 of~\cite{MR1417938}.)
We then calculate the remainders $\overline{m_\alpha}^{G_j}$ 
of the monomials $m_\alpha$ by these Gr\"obner bases $G_j$.
 An element $\sum_{\alpha} u_\alpha m_\alpha$ of $V_{d(v)}$ 
 with $u_\alpha\in \F_{25}$ is contained in
 $\Gamma(X, \LLL_{v})$ if and only if
 $$
 \sum_{\alpha} u_\alpha \overline{m_\alpha}^{G_j}=0
 \quad\textrm{for each $j\in J$}.
 $$
 These equalities constitute  a system of linear equations with unknowns $u_\alpha$.
 Solving these equations,
 we obtain a basis of $\Gamma(X, \LLL_{v})$
 as a list of polynomials in $V_{d(v)}$.
 \par
 \medskip
 Let $k$ be a positive integer.
 Then we can write the vector $kv\in \NS(X)$ as 
 $$
 kv:=kd(v) \theh- \sum_{j\in J} kc_j [\ell_j]
 $$
 using the same $d(v)$ and $J$ that appeared in~\eqref{eq:dv}.
 Under this choice,
 the natural homomorphism 
 $$
 \Gamma(X, \LLL_{v})\sp{\otimes k}\to \Gamma(X, \LLL_{kv})
 $$
 is given by restricting the linear homomorphism 
 $$
 g_1\otimes \dots\otimes g_k\;\mapsto\; \overline{ g_1\cdots g_k}^F
 $$
 from $V_{d(v)}\sp{\otimes k}$ to $V_{k d(v)}$.
\subsection{The morphisms $\Phi_h$ and $\phi_h$}\label{subsec:morphism}
We describe an algorithm that takes a vector $h\in \Pol_2(X)$ as input,
and calculates the morphisms $\Phi_h$, $\phi_h$ and a defining equation  
$$
w^2=s_h(x,y,z)
$$
of $X_h$ in $\P(3,1,1,1)$.
%
 %
 %
We have 
$$
\dim \Gamma (X, \LLL_h)=3,
\quad
\dim \Gamma (X, \LLL_{3h})=11,
\quad
\dim \Gamma (X, \LLL_{6h})=38.
$$
 We find an expression $h=d(h) \theh- \sum_{j\in J} c_j [\ell_j]$ of $h$
 in the form~\eqref{eq:dv}.
By the method described above,
we obtain three polynomials
$$
\xi_i(w,x,y)\in V_{d(h)} \qquad (i=0,1,2)
$$
that form a basis of $\Gamma (X, \LLL_h)$. 
The rational map $(w, x, y)\mapsto [\xi_0:\xi_1: \xi_2]$
gives the morphism $\Phi_h: X_F\to \P^2$.
\par
\smallskip
Next we calculate  eleven  polynomials
that form a basis of $\Gamma (X, \LLL_{3h})\subset V_{3d(h)}$
using the expression 
$3h=3d(h) \theh- \sum_{j\in J} 3c_j [\ell_j]$.
We compute the normal forms
$$
\overline{\xi_{i}\xi_{i\sprime}\xi_{i\spprime}}^F\qquad (i, i\sprime, i\spprime \in \{0,1,2\})
$$
of the ten polynomials $\xi_{i}\xi_{i\sprime}\xi_{i\spprime}$.
These normal forms are contained in $\Gamma (X, \LLL_{3h})$.
Then we  find a polynomial  $\omega\in V_{3d(h)}$ 
that is contained in  $\Gamma (X, \LLL_{3h})$, but is \emph{not} contained in 
the $10$-dimensional subspace spanned by $\overline{\xi_{i}\xi_{i\sprime}\xi_{i\spprime}}^F$.
The rational map 
$$
(w, x, y)\mapsto [\omega: \xi_0:\xi_1: \xi_2]\in \P(3,1,1,1)
$$
gives the morphism $\phi_h: X_F\to X_h$.
\par
\smallskip
We then compute
the $39$ normal forms
$$
\overline{\omega^2}^F,\qquad
\overline{\omega \xi_{i}\xi_{i\sprime}\xi_{i\spprime}}^F,
\qquad
\overline{\xi_{i_1}\xi_{i_2}\cdots\xi_{i_6}}^F
\quad(i, i\sprime, i\spprime, i_1, \dots, i_6\in \{0,1,2\}),
$$
which are contained in $\Gamma(X, \LLL_{6h})\subset V_{6d(h)}$.
Since $\dim\Gamma(X, \LLL_{6h})=38$,
there exists a non-trivial linear relation over $\F_{25}$ among these $39$ polynomials.
Using  homogeneous polynomials $b(x, y, z)$ of degree $3$
and $c(x, y, z)$ of degree $6$
with coefficients in $\F_{25}$,
we write this linear relation as
\begin{equation}\label{eq:linrel}
\overline{a\,\omega^2+ b(\xi_0,\xi_1, \xi_2)\,\omega+c(\xi_0,\xi_1, \xi_2)}^F=0,
\end{equation}
where $a\in \F_{25}$.
Since  $\omega$ is not invariant under the deck-transformation of  $X_F$ over $\P^2$,
we may assume that $a=1$.
We  replace
$\omega$ by
$$
\omega-2\,\overline{b(\xi_0,\xi_1, \xi_2)}^F\in V_{3d(h)}.
$$
Then the linear relation~\eqref{eq:linrel} is  written as 
$$
\overline{\omega^2}^F=\overline{s_h(\xi_0,\xi_1, \xi_2)}^F,
\quad\textrm{where}\quad
s_h(x, y, z):=-\,b(x, y, z)^2-c(x, y, z).
$$
The projective model $\psi_h: X_h\to \P^2$ is defined by $w^2=s_h(x, y, z)$.
%
\begin{remark}\label{rem:difficulty}
The computational difficulty of this method
grows rapidly as $d(h)$ increases.
\end{remark}
\subsection{The projective equivalence}\label{subsec:projmodels} 
Let $\algcloF$ denote an algebraic closure of $\F_{25}$.
For $T\in \GL_3(\algcloF)$,
we denote by $[T]\in \PGL_3(\algcloF)$ the image of $T$
by the natural map $\GL_3(\algcloF)\to \PGL_3(\algcloF)$,
and by $P\mapsto P^{[T]}$ the  linear transformation of $\P^2$ given by
$[a:b:c]\mapsto [a:b:c]\,T$.
Let $\HH_6$ denote the set of homogeneous polynomials of degree $6$ in variables $x,y,z$
with coefficients in $\F_{25}$.
For $f\in \HH_6\tensor \algcloF$,
we put
$$
f^T(x, y, z):=f(x\sprime, y\sprime, z\sprime),
\quad\textrm{where}\quad (x\sprime, y\sprime, z\sprime)=(x, y, z)\,T\inv.
$$
If $f=0$ defines a curve $C\subset \P^2$,
then $f^T=0$ defines the image $C^{[T]}$ of the curve $C$ by 
the projective linear transformation $P\mapsto P^{[T]}$.
\par
\medskip
Let $h$ and $h\sprime$ be elements of $\Pol_2(X)$.
By definition, 
we have the following:
\begin{equation}\label{eq:simdef1}
h\sim h\sprime \;\Longleftrightarrow \;
\parbox{9cm}{there exist $T\in \GL_3(\algcloF)$ and 
$c\in \algcloF\sptimes$ such that $s_{h\sprime}=c\,s_h^T$.}
\end{equation}
The polynomials $\omega, \xi_0,\xi_1, \xi_2$
giving $\phi_h :X_F\to X_h$ that are
obtained in the previous subsection are unique
up to the following transformations:
\begin{eqnarray*}
\omega &\mapsto &\lambda\omega, \quad\textrm{where}\quad     \lambda\in \F_{25}\sptimes, \\
(\xi_0,\xi_1, \xi_2)&\mapsto& (\xi_0,\xi_1, \xi_2) T,  \quad\textrm{where}\quad   T\in \GL_3(\F_{25}).
\end{eqnarray*}
Under this transformation,
the sextic polynomial $s_h\in \HH_6$  is changed to $\lambda^2 s_h^T$.
Therefore we can define the following relation $\sim_{\F}$ on $\Pol_2(X)$:
\begin{equation}\label{eq:simFdef1}
h\sim_{\F} h\sprime \;\Longleftrightarrow\;
\parbox{9cm}{there exist $T\in \GL_3(\F_{25})$ and $\lambda\in \F_{25}\sptimes$ such that $s_{h\sprime}=\lambda^2 s_h^T$.}
\end{equation}
We investigate the relation between $\sim$ and $\sim_{\F}$.
\begin{lemma}\label{lem:square}
Suppose that there exist $T\in \GL_3(\F_{25})$ and 
$c \in \algcloF\sptimes$  that satisfy $s_{h\sprime}=c\, s_h^T$.
Then $h\sim_{\F} h\sprime$ holds.
\end{lemma}
\begin{proof}
Let $K$ denote the quotient field of 
the integral domain $\F_{25}[w,x,y]/(F)$.   
Then we have $\algcloF\cap K=\F_{25}$.
By the assumption $s_{h\sprime}=c\, s_h^T$,
we see that $c \in \F_{25}\sptimes$ and that 
there exist non-zero elements $\omega$ and $\omega\sprime$ of $K$ such that
$\omega^{\prime 2}=c\,\omega^2$.
Hence $c$ is a non-zero square in $\F_{25}$.
\end{proof}
Let $B_1=\{f_1=0\}$ and $B_2=\{f_2=0\}$ be reduced plane curves 
defined by $f_1\in \HH_6$ and $f_2\in \HH_6$, respectively.
We consider the set 
$$
\isoms(B_1, B_2):=\set{\tau\in \PGL_3(\algcloF)}{B_1^\tau=B_2}
$$
of projective isomorphisms from $B_1$ to $B_2$ defined 
over $\algcloF$.
By definitions and Lemma~\ref{lem:square}, we have 
\begin{eqnarray}
h\sim h\sprime &\Longleftrightarrow & \isoms(B_h, B_{h\sprime})\ne \emptyset, \label{eq:relisoms1}
\\ 
h\sim_{\F} h\sprime &\Longleftrightarrow & 
\isoms(B_h, B_{h\sprime})\cap\PGL_{3}(\F_{25})\ne \emptyset. \label{eq:relisoms2}
\end{eqnarray}
\begin{definition}\label{def:tauQQ}
Let $Q=[Q_0, Q_1, Q_2, Q_3]$ and 
$Q\sprime=[Q\sprime_0, Q\sprime_1, Q\sprime_2, Q\sprime_3]$ 
be two ordered $4$-tuples of points of $\P^2$
such that no three points of $Q$ are colinear and no three points of $Q\sprime$ are colinear.
Then there exists a unique projective transformation
$\tau_{QQ\sprime}\in \PGL_3(\algcloF)$
such that
$$
Q^{\tau_{QQ\sprime}}:=[Q_0^{\tau_{QQ\sprime}}, Q_1^{\tau_{QQ\sprime}}, Q_2^{\tau_{QQ\sprime}}, Q_3^{\tau_{QQ\sprime}}]
$$
is equal to $Q\sprime$.
Let $T_{QQ\sprime}\in \GL_3(\algcloF)$
denote a matrix such that $[T_{QQ\sprime}]=\tau_{QQ\sprime}$.
\end{definition}
Let $B$ be a reduced plane curve  defined over $\algcloF$.
We define $\QQQ(B)$ to  be 
\begin{equation}\label{eq:QQQ}
\left\{\;\;[Q_0, Q_1, Q_2, Q_3]\;\;\left|\;\;
\parbox{6.2cm}{$Q_i\in \Sing (B)$ for $i=0, \dots, 3$, and no three  of $Q_0, \dots, Q_3$ are colinear}\right.\;\;\right\}.
\end{equation}
Let $R$ be an element of $\QQQ(B_1)$.
Then the map
$\tau\mapsto R^\tau$
induces a bijection
\begin{equation}\label{eq:isomstoQ}
\isoms (B_1, B_2)\;\cong\;
\set{Q\sprime \in \QQQ(B_2)}%
{\textrm{$f_2=c\,f_1^{T_{RQ\sprime}}$ for some $c\in \algcloF\sptimes$}}.
\end{equation}
If all points of $Q$ and $Q\sprime$ are $\F_{25}$-rational,
then we have $\tau_{QQ\sprime}\in \PGL_3(\F_{25})$.
Hence 
we obtain the following:
\begin{lemma}\label{lem:QQQ}
Suppose that
every singular point of $B_h$ and $B_{h\sprime}$ is $\F_{25}$-rational,
and that $\QQQ(B_h)$ and  $\QQQ(B_{h\sprime})$ are non-empty.
Then $\isoms(B_h, B_{h\sprime})$ is contained in $\PGL_3(\F_{25})$.
\qed
\end{lemma}
The bijection~\eqref{eq:isomstoQ} also 
provides  us with a practical method to calculate the group $\aut(B)=\isoms(B, B)$
for a  plane curve $B$ defined over $\F_{25}$ satisfying
$\Sing(B)\subset\P^2(\F_{25})$ and $\QQQ(B)\ne\emptyset$.
\section{Proof of Theorem~\ref{thm:projmodels}}\label{sec:proof1}
\algostep
First note that $\Pol_2(X)\cap \Ball_3=\{\theh\}$.
%
\algostep
We calculate the sets
$$
\VVV_{\delta}:=\set{v\in \NS(X)}{\intfNS{v}{v}=2,\;\;\intfNS{v}{\theh}=\delta},
$$
for $\delta=4$ and $5$ by the algorithm in Section~\ref{subsec:applicationI}.
The cardinalities of  these sets are
$|\VVV_4|=1,020,600$ and 
$|\VVV_5|=208,059,000$.
We put
$$
\VVV:=\{\theh\}\cup \VVV_4\cup \VVV_5.
$$
Our goal is to calculate the subset  $\Pol_2(X)\cap \Ball_5=\Pol_2(X)\cap \VVV$ of $\VVV$,  and decompose it 
into the equivalence classes of the relation $\sim$
of the  projective equivalence.
Note that $\Aut(X, \theh)$ acts on $\VVV_4$, $\VVV_5$ and $\Pol_2(X)$,
and that, if  $h$ and $h\sprime$ are in the same $\Aut(X, \theh)$-orbit, 
then we have $h\sim_{\F}h\sprime$,
because every element of $\Aut(X, \theh)$ is defined over $\F_{25}$.
\algostep
We have embedded $\Aut(X, \theh)$ in  $\mathrm{O}(\NS(X))$
by~\eqref{eq:rep}.
Let $\vx=[x_1, \dots, x_{22}]$ and $\vy=[y_1, \dots, y_{22}]$ be vectors in $\NS(X)$.
We put
$$
\vx<_{\mathrm{lex}} \vy
\quad
\Longleftrightarrow\quad
\textrm{there exists $k$ such that 
$x_k<y_k$ and 
$x_j=y_j$ for  $j<k$},
$$
and define a total order $<$ on $\NS(X)$  by
$$
\vx< \vy
\quad
\Longleftrightarrow\quad\sum_{i=1}^{22} |x_i|< \sum_{i=1}^{22} |y_i| \quad \textrm{or}\quad 
\left(\sum_{i=1}^{22} |x_i|= \sum_{i=1}^{22} |y_i| \quand \vx<_{\mathrm{lex}} \vy\right). 
$$
We then denote by $\RRR$ the set of vectors $v \in \NS(X)$ that  
are minimal in the $\Aut(X, \theh)$-orbit containing $v$:
$$
\RRR:=\set{v\in \NS(X)}{v\le v T\;\;\textrm{for all}\;\; T\in \Aut(X, \theh)}.
$$
We  define the representative vector $v_{\orb}$ of 
each $\Aut(X,\theh)$-orbit $\orb\subset \NS(X)$ by
$$
\orb\cap \RRR=\{v_{\orb}\}.
$$
We calculate the list $\RRR\cap \VVV_4$, $\RRR\cap \VVV_5$, 
and the order of the stabilizer subgroup $\Stab(v)\subset \Aut(X, \theh)$
for each $v\in \RRR\cap \VVV$.
We  obtain
$$
|\RRR\cap \VVV_4|=|\VVV_4/\Aut(X, \theh)|=8
\quand
|\RRR\cap \VVV_5|=|\VVV_5/\Aut(X, \theh)|=312.
$$
\begin{remark}
We choose this total order $<$ on $\NS(X)$ so that we can express 
each $v\in \RRR\cap \VVV$ in
the form~\eqref{eq:dv}
with  $d(v)$ small.
See Remark~\ref{rem:difficulty}.
\end{remark}
%
%
\algostep
For each $v\in \RRR\cap \VVV$,
we calculate the $\Gal(\F_{25}/\F_{5})$-conjugate $\bar{v}=v\Gamma_{\NS}$
of $v$, 
where $\Gamma_{\NS}$ is the matrix that has been calculated in Section~\ref{sec:NSofX}, 
and find the representative vector $v^\Gamma\in \RRR\cap \VVV$
of the  $\Aut(X, \theh)$-orbit containing $\bar{v}$.
%
%
\algostep
For each $v\in \RRR\cap \VVV$,
we calculate the sets $S_1$ and $S_2$ in Proposition~\ref{prop:polarization},
and determine whether $v$ is a polarization or not.
We  obtain
$$
|\Pol_2(X)\cap\RRR\cap \VVV_4|=7\quand
|\Pol_2(X)\cap\RRR\cap \VVV_5|=224.
$$
%
%
%
%
%
%
%
%
\algostep
For simplicity, we put
$$
\HHH:=\Pol_2(X)\cap\RRR\cap \VVV.
$$
By means of the algorithms in Sections~\ref{subsec:hexceptional} and~\ref{subsec:hlines},
we calculate, for each  $h\in \HHH$,
the set $\Exc(h)$ of the classes of $h$-exceptional curves, 
and the set $\Lin(h)$ of the classes of $h$-lines.
From $\Exc(h)$, we  determine the $ADE$-type $\RT(h)$ of $\Sing (B_{h})$.
%
We then confirm that the union of $\Exc(h)$ and  $\Lin(h)$  spans $\NS(X)$
for any $h\in \HHH$.
Thus Proposition~\ref{prop:splgen} is proved.
%
%
%
\algostep
For each $h\in \HHH$,
we carry out the computation in Section~\ref{sec:morphism}, 
and  calculate  polynomials
$\omega, \xi_0, \xi_1, \xi_2\in \F_{25}[w,x,y]$
that give the morphism $\phi_h : X_F\to X_h$,
and  $s_h(x,y,z)\in \HH_6$ 
such that $w^2=s_h(x,y,z)$ defines $X_h$.
Then we compute the coordinates of the singular points of $B_h=\{s_h=0\}$.
\begin{remark}\label{rem:QQQSing}
By this computation, we observe the following fact.
For any $h\in \HHH$ with $\RT(h)\ne 0$,
every singular point of $B_h$
is $\F_{25}$-rational, 
and the set $\QQQ(B_h)$ defined by~\eqref{eq:QQQ} is non-empty.
By Lemma~\ref{lem:QQQ},
it follows that  
$\isoms(B_{h}, B_{h\sprime})$ is contained in $\PGL_3(\F_{25})$
for any $h, h\sprime\in \HHH$ with $\RT(h)\ne 0$ and $\RT(h\sprime)\ne 0$.
\end{remark}
\begin{remark}
It turns out that each $(-2)$-curve contracted by $\phi_h$
is either an $h_F$-line or an irreducible component of the pull-back by $\psi_F$ of
a plane conic totally tangent to $B_F$ (see~\cite{shimadaHermite}).
We can calculate the coordinates of the singular points of $B_h$ using this fact.
\end{remark}
\algostep
We decompose $\HHH$ into  
the equivalence classes under the relation $\sim_{\F}$ 
defined by~\eqref{eq:simFdef1},
and confirm that the relations $\sim$ and $\sim_{\F}$ are the same on $\HHH$.
\subsubsection{The case where $B_h$ is non-singular.}\label{subsubsec:nonsing}
In $\HHH$, there are exactly three polarizations $h$ such that $\RT(h)=0$: $\theh$ and 
\begin{eqnarray*}
h_F\sprime &=& 
[1, 0, 0, 1, 0, 1, 0, 0, 0, 0, 1, 0, 1, 0, -1, 0, 0, 0, 0, 0, 0, 0]\in\VVV_4,
\quand\\
h_F\spprime &=& 
[0, -1, 0, 2, 1, 0, 0, 0, 0, 0, 1, 0, 1, 0, 1, 1, 0, -1, 0, 0, 0, 0]\in\VVV_5.
\end{eqnarray*}
Applying the following result,
which is a corollary of n.~3 of \cite{MR0213949},   to $s_{h_F\sprime}$ and $s_{h_F\spprime}$,
we see that $h_F\sprime\sim_{\F} \theh$ and $h_F\spprime\sim_{\F} \theh$.
\begin{corollary}
For $h\in \Pol_2(X)$,
we have $h\sim_{\F}\theh$ if and only if 
there exist a $3\times 3$ non-degenerate matrix $(a_{ij})$ over $\F_{25}$ with $a_{ij}=a_{ji}^5$ 
and  $\lambda\in \F_{25}\sptimes$ such that 
$s_h=s_h(x_0, x_1, x_2)$ is of the form
$\lambda^2\sum_{i, j=0}^2 a_{ij}x_i x_j^5$.
\end{corollary}
\subsubsection{The case where $B_h$ is singular.}\label{subsubsec:sing}
We introduce a total order $\prec$ on the set $\HH_6$.
(Any total order will do.)
We fix four reference points 
$$
P_0:=[1:0:0],
\;\;
P_1:=[0:1:0],
\;\;
P_2:=[0:0:1],
\;\;
P_3:=[1:1:1],
$$
and put $P:=[P_0, P_1, P_2, P_3]$.
For $h\in \HHH$ with $\RT(h)\ne 0$, 
we put 
\begin{eqnarray*}
\TTT(h)&:=&\set{\tau\in \PGL_3(\algcloF)}{\Sing(B_h^\tau)\ni P_i\;\;\textrm{for}\;\; i=0,1,2,3}
=\shortset{\tau_{QP}}{Q\in \QQQ(B_h)},\\
S(h)&:=&\set{\lambda^2 s_h^T }{\lambda\in \F_{25}\sptimes,\;\; T\in \GL_3(\F_{25})},\\
S^P(h)&:=&
\set{s\sprime_h\in S(h)}{\textrm{the curve $s_h\sprime=0$ is singular at $P_0, \dots, P_3$}}.
\end{eqnarray*}
By Remark~\ref{rem:QQQSing}, we have $\TTT(h)\subset\PGL_3(\F_{25})$ and $\TTT(h)\ne\emptyset$, and hence 
$$
S^P(h)=\set{\lambda^2 s_h^T}{\lambda\in \F_{25}\sptimes,\;\; T\in \GL_3(\F_{25}),\;\; [T]\in \TTT(h)}\ne \emptyset
$$
holds.
Since $\QQQ(B_h)$ is easily calculated, so is $S^P(h)$.
We put
$$
s_h^{\min} := \textrm{the minimal element of $S^P(h)$ with respect to the fixed total order $\prec$}.
$$
By definition, we have $h\sim_{\F}h\sprime$ if and only if $S(h)=S(h\sprime)$.
Hence we have 
\begin{equation*}\label{eq:projmodelS}
h\sim_{\F}h\sprime\;\; \Longleftrightarrow\;\;s_h^{\min}=s_{h\sprime}^{\min}.
\end{equation*}
By this method, we  decompose $\HHH$ into the equivalence classes of  $\sim_{\F}$.
\par
\medskip
Remark~\ref{rem:QQQSing} combined with~\eqref{eq:relisoms1},~\eqref{eq:relisoms2} 
imply that  $\sim$ and $\sim_{\F}$ define the same relation on $\HHH$.
Thus the equivalence classes $\projmodel_0$, \dots, $\projmodel_{64}$ of $\sim$ are obtained.
\par
\medskip
For $h\in \HHH$,
we denote by $[h]\subset \HHH$ the equivalence class of $\sim$
containing $h$,  by $s_{[h]}$ 
the polynomial  $s_h^{\min}$ obtained above, and 
 by $B_{[h]}$ the plane curve $\{s_{[h]}=0\}$.
%
%
\algostep
For each equivalence class $[h]\subset \HHH$, 
we calculate the group $\aut (B_{[h]})=\isoms(B_{[h]}, B_{[h]})$  
and the set $\isoms(B_{[h]}, \overline{B_{[h]}})$ by the method given in Section~\ref{subsec:projmodels},
where $\overline{B_{[h]}}$ is the plane curve defined by the polynomial 
$\overline{s_{[h]}}\in  \HH_6$
obtained from $s_{[h]}$ by $\sqrt{2}\mapsto-\sqrt{2}$.
%
%
%
%
%
%
%
\algostep
We search for $(T, \lambda)\in \GL_3(\F_{25})\times \F_{25}\sptimes$
such that $\lambda^2 s_{[h]}^T$ has coefficients in $\F_{5}$.
If such $(T, \lambda)$ exists, then we necessarily have $h\sim h^\Gamma$.
\begin{proposition}\label{prop:F5}
For $f\in \HH_6$,
the following conditions are equivalent.
\rm{(i)}
There exist $T\in \GL_3(\F_{25})$
and $\lambda\in \F_{25}\sptimes$
such that $\lambda^2 f^T$ has coefficients in $ \F_{5}$.
\rm{(ii)}
There exist $M\in \GL_3(\F_{25})$
and $c\in \F_{25}\sptimes$
such that $f^M=c\,\bar{f}$, $M\overline M=\Id_3$ and $c^3=1$.
\end{proposition}
Since we have already calculated the set $\isoms(B_{[h]}, \overline{B_{[h]}})$
for every $[h]\subset \HHH$,
we can make the list of $(M, c) \in \GL_3(\F_{25})\times \F_{25}\sptimes$
such that $s_{[h]}^M=c\,\overline{s_{[h]}}$.
Therefore we can determine whether the condition (ii) is satisfied or not
for $f=s_{[h]}$.
The proof below shows how to
find $(T, \lambda)$ in the condition (i) 
from $(M, c)$ in the condition (ii).
\begin{proof}[Proof of Proposition~\ref{prop:F5}]
Suppose that (i) holds.
Since  $\bar{\lambda}^2\bar{f}^{\,\overline{T}}=\lambda^2 f^T$,
we have $(\lambda\inv \bar{\lambda})^2 \bar{f}=f^{T\overline{T}\inv}$.
Then $M:=T\overline{T}\inv$ and $c:=(\lambda\inv \bar{\lambda})^2=\lambda^8$
satisfy the equalities in (ii).
Conversely, suppose that (ii) holds.
Then there exists $T\in \GL_3(\F_{25})$ such that $M=T\overline{T}\inv$.
Indeed, let $m: \F_{25}^{3}\to\F_{25}^{3}$ be defined by
$m(\vx):=\vx+\bar{\vx} M$,
where vectors of  $\F_{25}^{3}$ are written as row vectors.
Then there exist $\vx_1, \vx_2, \vx_3$ such that 
$m(\vx_1), m(\vx_2), m(\vx_3)$ are linearly independent.
Let $C$ denote the $3\times 3$ matrix 
whose row vectors are $\vx_1, \vx_2, \vx_3$.
We put
$$
S:=C+\overline{C}M,
$$
which is non-degenerate.
Then we have $\overline{S}=SM\inv$.
Therefore, putting $T:=\overline{S}\inv$, we have $M=T\overline{T}\inv$.
Since $f^M=c\,\bar{f}$,
we have $f^T=c\,\overline{f^T}$.
Since $c^3=1$,
there exists $\lambda\in \F_{25}\sptimes$ such that
$c=\lambda^8=(\lambda\inv\bar{\lambda})^2$.
Then we have
$\lambda^2 f^T=\bar{\lambda}^2 \overline{f^T}$, 
and hence $\lambda^2 f^T$ has coefficients in $\F_{5}$.
\end{proof}
\begin{remark}
Except for the equivalence class $\projmodel_7=\overline{\projmodel}_7$,
we have found a defining equation $s_{\F, [h]}$ of $B_h$ with coefficients in $\F_5$
for each $\projmodel_n$ with $\projmodel_n=\overline{\projmodel}_n$.
\end{remark}
\section{The list of projective models $\projmodel_0$, \dots, $\projmodel_{64}$}\label{sec:list}
%
%
%
%
%
%
{\tiny
\PPMM{0}{0}{0}{378000}{[0,0,0,0,0,0,0,0,0,126]}{x^6+y^6+1}{13051}{[756000]_2, [720]_4, [63]_5}
{[1,1,0,0,0,0,0,0,0,0,0,0,0,0,0,0,0,0,0,0,0,0]}
%
%
\PPMM{1}{1}{6A\sb{1}}{12}{[0, 0, 0, 0, 0, 0, 12, 0, 30, 18]}{{x}^{6}+3\,{x}^{5}y+{x}^{4}{y}^{2}+2\,{x}^{3}{y}^{3}+{y}^{6}+3\,{x}^{4
}+3\,{x}^{2}{y}^{2}+x{y}^{3}+3\,xy+2\,{y}^{2}+4
}{5607000}{[3, 12]_4,[1, 1, 1, 1, 1, 1, 2, 2]_5}{[0, 0, 0, 0, 0, 0, 0, 0, 1, 1, 0, 0, 0, 0, 1, 0, 0, 0, 0, 0, 0, 1]}
\PPMM{2}{2}{7A\sb{1}}{6}{[0, 0, 0, 0, 0, 0, 15, 0, 24, 13]}{{x}^{6}+2\,{x}^{4}{y}^{2}+{x}^{2}{y}^{4}+{x}^{2}{y}^{3}+2\,{y}^{5}+{x}
^{4}+2\,{y}^{4}+2\,{x}^{2}y+2\,{y}^{3}+3\,{y}^{2}+3\,y+2
}{6678000}{[2]_4,[1, 1, 1, 1, 1, 1, 2, 2, 2, 2, 6, 6]_5}{[0, 0, 0, 0, 0, 0, 0, 0, 0, 1, 1, 1, 0, 0, 0, 0, 1, 0, 0, 0, 0, 0]}
\PPMM{3}{3}{3A\sb{1} + 2A\sb{2}}{6}{[0, 0, 0, 1, 0, 6, 0, 12, 15, 16]}{{x}^{6}+3\,{x}^{3}{y}^{3}+{y}^{6}+3\,{x}^{3}y+2\,{y}^{2}+2
}{2268000}{[1, 1, 1]_5}{[0, 0, 0, 0, 0, 0, 0, 0, 1, 1, 1, 0, 1, 0, 1, 0, 0, 0, 0, 0, 0, 0]}
\PPMM{4}{4}{8A\sb{1}}{8}{[0, 0, 0, 0, 0, 0, 16, 0, 24, 4]}{{x}^{6}+3\,{x}^{4}{y}^{2}+{x}^{2}{y}^{4}+4\,{x}^{2}{y}^{3}+4\,{y}^{5}+
{x}^{4}+2\,{x}^{2}{y}^{2}+3\,{y}^{4}+2\,{x}^{2}y+4\,{x}^{2}+{y}^{2}+4
\,y
}{2457000}{[4]_4,[1, 2, 2, 2, 2]_5}{[0, 0, 0, 0, 1, 0, 0, 0, 1, 0, 0, 0, 0, 1, 1, 0, 0, 0, 0, 0, 0, 0]}
\PPMM{5}{5}{8A\sb{1}}{4}{[0, 0, 0, 5, 0, 0, 12, 0, 20, 8]}{{x}^{4}{y}^{2}+{x}^{2}{y}^{4}+2\,{x}^{4}+4\,{x}^{2}{y}^{2}+{y}^{4}+{x}
^{2}+4\,{y}^{2}+4
}{2268000}{[1, 1, 2, 2]_5}{[0, 0, 0, 0, 0, 1, 0, 0, 1, 0, 0, 0, 1, 0, 0, 1, 0, 0, 1, 0, 0, 0]}
\PPMM{6}{6}{6A\sb{1} + A\sb{2}}{6}{[0, 0, 0, 2, 0, 6, 9, 0, 12, 14]}{{x}^{6}+4\,{x}^{4}{y}^{2}+2\,{x}^{2}{y}^{4}+2\,{x}^{2}y+{y}^{3}+4
}{1512000}{[1, 1]_5}{[0, 0, 0, 0, 0, 0, 0, 0, 0, 1, 0, 0, 0, 0, 1, 0, 1, 1, 0, 1, 0, 0]}
\PPMM{7}{7}{6A\sb{1} + A\sb{2}}{2}{[0, 0, 0, 2, 0, 4, 9, 4, 13, 12]}{\sqrt {2}{x}^{3}{y}^{3}+ \left( 1+3\,\sqrt {2} \right) {x}^{2}{y}^{4}+
{x}^{4}+ \left( 2+2\,\sqrt {2} \right) {x}^{3}y+ \left( 1+4\,\sqrt {2}
 \right) {x}^{2}{y}^{2}+x{y}^{3}+ \left( 2+2\,\sqrt {2} \right) {y}^{4
}+\sqrt {2}{x}^{2}+ \left( 1+3\,\sqrt {2} \right) xy
}{4914000}{[1, 1, 1, 1, 1, 1, 2]_5}{[0, 0, 0, 0, 0, 0, 0, 0, 1, 0, 0, 0, 0, 0, 1, 0, 0, 1, 0, 1, 0, 1]}
\PPMM{8}{8}{6A\sb{1} + A\sb{2}}{1}{[0, 0, 0, 1, 0, 4, 9, 4, 20, 6]}{{x}^{6}+2\,{x}^{5}y+{x}^{4}{y}^{2}+3\,{x}^{5}+2\,x{y}^{4}+{x}^{3}y+3\,
{x}^{2}{y}^{2}+4\,x{y}^{2}+{y}^{3}+3\,{y}^{2}+3\,x+3\,y
}{9828000}{[1, 1, 1, 1, 1, 1, 1, 1, 1, 1, 1, 1, 1]_5}{[0, 0, 0, 0, 0, 0, 0, 0, 0, 1, 0, 0, 1, 0, 0, 0, 1, 1, 0, 0, 0, 1]}
\PPMM{9}{10}{4A\sb{1} + 2A\sb{2}}{2}{[0, 0, 0, 0, 0, 6, 5, 10, 13, 9]}{{x}^{5}y+ \left( 2+\sqrt {2} \right) {x}^{4}{y}^{2}+ \left( 1+4\,
\sqrt {2} \right) {x}^{3}{y}^{3}+ \left( 3+\sqrt {2} \right) {x}^{2}{y
}^{4}+ \left( 2+4\,\sqrt {2} \right) x{y}^{5}+ \left( 2+\sqrt {2}
 \right) {y}^{6}+ \left( 2+3\,\sqrt {2} \right) {x}^{4}+ \left( 1+4\,
\sqrt {2} \right) {x}^{3}y+ \left( 3+\sqrt {2} \right) {y}^{4}+
 \left( 1+4\,\sqrt {2} \right) {x}^{2}+ \left( 3+\sqrt {2} \right) xy+
3\,{y}^{2}+2+3\,\sqrt {2}
}{4158000}{[1, 1, 1, 1, 1, 2]_5}{[0, 0, 0, 0, 0, 0, 0, 0, 0, 1, 0, 0, 1, 0, 1, 1, 0, 0, 0, 0, 0, 1]}
\PPMM{11}{11}{9A\sb{1}}{54}{[0, 0, 0, 9, 0, 0, 0, 0, 27, 0]}{{x}^{6}+4\,{x}^{3}{y}^{3}+4\,{y}^{6}+{x}^{4}+4\,x{y}^{3}+3\,{x}^{2}+4
}{84000}{[9]_5}{[0, 0, 0, 0, 0, 0, 0, 0, 1, 1, 0, 0, 0, 1, 1, 0, 1, 1, -1, 0, 0, 0]}
\PPMM{12}{12}{9A\sb{1}}{9}{[0, 0, 0, 0, 0, 0, 18, 0, 18, 3]}{4\,{x}^{4}{y}^{2}+3\,{x}^{2}{y}^{4}+4\,{y}^{6}+{x}^{5}+3\,{x}^{3}{y}^{
2}+2\,x{y}^{4}+{x}^{4}+2\,{x}^{2}{y}^{2}+4\,x{y}^{3}+2\,x{y}^{2}+4\,{y
}^{3}+4\,{x}^{2}+2\,xy+1
}{1596000}{[9]_4,[1, 1]_5}{[0, 0, 0, 0, 0, 0, 0, 0, 1, 1, 0, 1, 0, 0, 1, 0, 0, 0, 0, 0, 0, 0]}
\PPMM{13}{14}{9A\sb{1}}{6}{[0, 0, 0, 6, 0, 0, 12, 0, 15, 5]}{\sqrt {2}{x}^{5}y+2\,{x}^{4}{y}^{2}+ \left( 3+2\,\sqrt {2} \right) {x}
^{3}{y}^{3}+ \left( 4+2\,\sqrt {2} \right) {x}^{2}{y}^{4}+ \left( 4+4
\,\sqrt {2} \right) x{y}^{5}+\sqrt {2}{y}^{6}+ \left( 1+\sqrt {2}
 \right) {x}^{4}+ \left( 4+3\,\sqrt {2} \right) {x}^{3}y+ \left( 1+4\,
\sqrt {2} \right) {x}^{2}{y}^{2}+ \left( 1+4\,\sqrt {2} \right) {y}^{4
}+ \left( 3+3\,\sqrt {2} \right) {x}^{2}+ \left( 1+\sqrt {2} \right) x
y+ \left( 3+4\,\sqrt {2} \right) {y}^{2}+1+\sqrt {2}
}{882000}{[1, 6]_5}{[0, 0, 0, 0, 1, 0, 0, 0, 0, 1, 0, 0, 1, 0, 0, 0, 0, 0, 0, 0, 1, 1]}
\PPMM{15}{16}{9A\sb{1}}{3}{[0, 0, 0, 3, 0, 0, 18, 0, 12, 6]}{ \left( 2+2\,\sqrt {2} \right) {x}^{2}{y}^{4}+{x}^{4}y+ \left( 4+4\,
\sqrt {2} \right) {x}^{3}{y}^{2}+ \left( 1+\sqrt {2} \right) {x}^{2}{y
}^{3}+ \left( 2+4\,\sqrt {2} \right) x{y}^{4}+ \left( 1+\sqrt {2}
 \right) {x}^{4}+ \left( 1+2\,\sqrt {2} \right) {x}^{3}y+ \left( 2+3\,
\sqrt {2} \right) x{y}^{3}+ \left( 2+4\,\sqrt {2} \right) {x}^{2}y+
 \left( 2+\sqrt {2} \right) x{y}^{2}+ \left( 2+\sqrt {2} \right) xy+2
\,{y}^{2}
}{2268000}{[1, 1, 3, 3, 3]_5}{[0, -1, 0, 0, 0, 0, 1, 1, 0, 0, 1, 0, 0, 1, 0, 0, 1, 0, 0, 1, 0, 0]}
\PPMM{17}{17}{9A\sb{1}}{2}{[0, 0, 0, 5, 0, 0, 12, 0, 17, 4]}{{x}^{5}y+2\,{x}^{4}{y}^{2}+4\,{x}^{3}{y}^{3}+2\,{x}^{2}{y}^{4}+4\,x{y}
^{5}+3\,{y}^{6}+2\,{x}^{2}{y}^{2}+2\,{x}^{2}+xy
}{3402000}{[1, 1, 1, 2, 2, 2]_5}{[1, 0, 0, 0, 0, 0, 0, 0, 1, 0, 0, 0, 0, 0, 1, 0, 0, 0, 0, 1, 0, 1]}
\PPMM{18}{19}{7A\sb{1} + A\sb{2}}{2}{[0, 0, 0, 1, 0, 4, 13, 3, 12, 4]}{\sqrt {2}{x}^{4}{y}^{2}+ \left( 1+2\,\sqrt {2} \right) {x}^{3}{y}^{3}+
 \left( 3+4\,\sqrt {2} \right) {x}^{2}{y}^{4}+3\,\sqrt {2}x{y}^{5}+
 \left( 2+2\,\sqrt {2} \right) {x}^{4}+\sqrt {2}{x}^{3}y+4\,{x}^{2}{y}
^{2}+3\,\sqrt {2}x{y}^{3}+ \left( 2+2\,\sqrt {2} \right) {y}^{4}+
 \left( 1+\sqrt {2} \right) {x}^{2}+4\,\sqrt {2}xy+ \left( 1+\sqrt {2}
 \right) {y}^{2}+2+2\,\sqrt {2}
}{3024000}{[1, 1, 1, 1]_5}{[0, 0, 0, 0, 0, 0, 0, 0, 1, 0, 0, 0, 1, 1, 1, 1, 0, 0, 0, 0, 0, 0]}
\PPMM{20}{21}{7A\sb{1} + A\sb{2}}{1}{[0, 0, 0, 3, 0, 4, 8, 3, 13, 6]}{2\,\sqrt {2}{x}^{3}{y}^{3}+ \left( 3+\sqrt {2} \right) {x}^{2}{y}^{4}+
{x}^{4}y+ \left( 4+2\,\sqrt {2} \right) {x}^{3}{y}^{2}+ \left( 3+4\,
\sqrt {2} \right) {x}^{2}{y}^{3}+ \left( 4+4\,\sqrt {2} \right) x{y}^{
4}+{x}^{4}+3\,\sqrt {2}{x}^{2}{y}^{2}+3\,\sqrt {2}x{y}^{3}+4\,{y}^{4}+
\sqrt {2}{x}^{3}+2\,\sqrt {2}{x}^{2}y+\sqrt {2}x{y}^{2}+ \left( 2+2\,
\sqrt {2} \right) {y}^{3}+3\,{x}^{2}+ \left( 3+2\,\sqrt {2} \right) xy
+ \left( 2+3\,\sqrt {2} \right) {y}^{2}
}{5292000}{[1, 1, 1, 1, 1, 1, 1]_5}{[0, 0, 0, 0, 0, 0, 0, 0, 0, 1, 0, 0, 0, 0, 1, 0, 0, 1, 0, 1, 0, 1]}
\PPMM{22}{23}{7A\sb{1} + A\sb{2}}{1}{[0, 0, 0, 2, 0, 4, 9, 3, 16, 4]}{{x}^{3}{y}^{3}+ \left( 1+3\,\sqrt {2} \right) {x}^{2}{y}^{4}+{x}^{4}y+
 \left( 3+2\,\sqrt {2} \right) {x}^{3}{y}^{2}+3\,\sqrt {2}{x}^{2}{y}^{
3}+ \left( 2+4\,\sqrt {2} \right) x{y}^{4}+\sqrt {2}{x}^{4}+ \left( 2+
4\,\sqrt {2} \right) {x}^{3}y+4\,x{y}^{3}+ \left( 1+3\,\sqrt {2}
 \right) {y}^{4}+ \left( 2+\sqrt {2} \right) {x}^{3}+ \left( 3+3\,
\sqrt {2} \right) {x}^{2}y+\sqrt {2}{y}^{3}+ \left( 4+2\,\sqrt {2}
 \right) {x}^{2}+4\,\sqrt {2}xy+ \left( 1+4\,\sqrt {2} \right) {y}^{2}
}{5292000}{[1, 1, 1, 1, 1, 1, 1]_5}{[0, 0, 0, 0, 0, 0, 0, 0, 0, 1, 1, 0, 1, 0, 0, 0, 0, 1, 0, 0, 0, 1]}
\PPMM{24}{24}{5A\sb{1} + 2A\sb{2}}{8}{[0, 0, 0, 2, 0, 10, 0, 0, 16, 8]}{{x}^{3}{y}^{3}+{x}^{4}+{x}^{2}{y}^{2}+{y}^{4}+xy
}{378000}{[2]_5}{[0, 0, 0, 0, 0, 0, 0, 0, 0, 1, 0, 1, 0, 0, 0, 0, 1, 1, 0, 1, 0, 0]}
\PPMM{25}{26}{5A\sb{1} + 2A\sb{2}}{2}{[0, 0, 0, 0, 0, 6, 9, 8, 6, 8]}{{x}^{2}{y}^{4}+{x}^{4}y+ \left( 1+\sqrt {2} \right) {x}^{3}{y}^{2}+
 \left( 3+4\,\sqrt {2} \right) {x}^{2}{y}^{3}+ \left( 3+2\,\sqrt {2}
 \right) x{y}^{4}+ \left( 1+\sqrt {2} \right) {x}^{3}y+ \left( 1+2\,
\sqrt {2} \right) {x}^{2}{y}^{2}+ \left( 3+\sqrt {2} \right) x{y}^{3}+
 \left( 1+4\,\sqrt {2} \right) {x}^{2}y+ \left( 1+2\,\sqrt {2}
 \right) x{y}^{2}+3\,{x}^{2}+4\,\sqrt {2}xy+ \left( 1+4\,\sqrt {2}
 \right) {y}^{2}
}{2268000}{[1, 1, 2, 2]_5}{[0, 0, 0, 0, 0, 0, 0, 0, 0, 1, 0, 0, 1, 0, 1, 0, 0, 1, 0, 0, 0, 1]}
\PPMM{27}{27}{5A\sb{1} + 2A\sb{2}}{1}{[0, 0, 0, 1, 0, 8, 4, 4, 13, 6]}{{x}^{6}+3\,{x}^{4}{y}^{2}+{x}^{2}{y}^{4}+{x}^{3}{y}^{2}+3\,{x}^{2}{y}^
{3}+x{y}^{4}+2\,{x}^{3}y+3\,x{y}^{3}+4\,{x}^{3}+3\,{x}^{2}y+4\,x{y}^{2
}+4\,{y}^{2}
}{3780000}{[1, 1, 1, 1, 1]_5}{[0, 0, 0, 0, 0, 0, 0, 0, 0, 1, 0, 0, 0, 0, 0, 0, 1, 1, 0, 1, 0, 1]}
\PPMM{28}{29}{5A\sb{1} + 2A\sb{2}}{1}{[0, 0, 0, 1, 0, 6, 4, 8, 14, 4]}{{x}^{4}{y}^{2}+ \left( 2+2\,\sqrt {2} \right) {x}^{3}{y}^{3}+ \left( 3
+2\,\sqrt {2} \right) {x}^{2}{y}^{4}+ \left( 1+\sqrt {2} \right) {x}^{
4}y+2\,\sqrt {2}{x}^{3}{y}^{2}+ \left( 2+\sqrt {2} \right) x{y}^{4}+
 \left( 2+3\,\sqrt {2} \right) {x}^{4}+4\,{x}^{2}{y}^{2}+ \left( 1+3\,
\sqrt {2} \right) {y}^{4}+ \left( 3+4\,\sqrt {2} \right) {x}^{3}+4\,
\sqrt {2}x{y}^{2}+ \left( 1+\sqrt {2} \right) {y}^{3}+ \left( 4+2\,
\sqrt {2} \right) {x}^{2}+ \left( 3+3\,\sqrt {2} \right) xy+ \left( 1+
2\,\sqrt {2} \right) {y}^{2}
}{4536000}{[1, 1, 1, 1, 1, 1]_5}{[0, 0, 0, 0, 0, 0, 0, 0, 0, 0, 0, 0, 1, 0, 1, 0, 0, 1, 0, 1, 0, 1]}
\PPMM{30}{31}{3A\sb{1} + 3A\sb{2}}{3}{[0, 0, 0, 0, 0, 6, 3, 15, 6, 6]}{{x}^{4}{y}^{2}+ \left( 1+\sqrt {2} \right) {x}^{3}{y}^{3}+ \left( 2+3
\,\sqrt {2} \right) {x}^{2}{y}^{4}+{x}^{4}y+4\,{x}^{3}{y}^{2}+ \left( 
3+3\,\sqrt {2} \right) {x}^{2}{y}^{3}+4\,\sqrt {2}x{y}^{4}+4\,{x}^{4}+
 \left( 2+3\,\sqrt {2} \right) {x}^{3}y+{x}^{2}{y}^{2}+ \left( 4+2\,
\sqrt {2} \right) {y}^{4}+ \left( 3+2\,\sqrt {2} \right) {x}^{3}+
 \left( 4+3\,\sqrt {2} \right) {x}^{2}y+ \left( 4+4\,\sqrt {2}
 \right) x{y}^{2}+ \left( 2+4\,\sqrt {2} \right) {y}^{3}+{x}^{2}+
\sqrt {2}xy+3\,{y}^{2}
}{1260000}{[1, 3, 3]_5}{[0, 0, 0, 0, 0, 0, 0, 0, 0, 0, 0, 1, 1, 0, 0, 0, 1, 1, 0, 0, 0, 1]}
\PPMM{32}{32}{10A\sb{1}}{20}{[0, 0, 0, 0, 0, 0, 20, 0, 10, 1]}{{x}^{6}+2\,{x}^{4}y+{y}^{5}+4\,{x}^{2}{y}^{2}+{y}^{3}+4\,{x}^{2}+4\,y
}{226800}{[20]_4,[4]_5}{[0, 0, 0, 0, 0, 0, 0, 1, 0, 0, 0, 0, 1, 0, 0, 0, 0, 1, 0, 0, 0, 1]}
\PPMM{33}{33}{10A\sb{1}}{4}{[0, 0, 0, 6, 0, 0, 16, 0, 4, 6]}{{x}^{6}+{x}^{4}{y}^{2}+3\,{x}^{3}{y}^{3}+3\,{x}^{2}{y}^{4}+2\,{y}^{6}+
{x}^{2}{y}^{2}+4\,xy+4
}{756000}{[2, 2]_5}{[0, 0, 0, 0, 1, 0, 0, 0, 0, 1, 1, 0, 0, 0, 0, 0, 1, 0, 0, 1, 0, 0]}
\PPMM{34}{35}{10A\sb{1}}{2}{[0, 0, 0, 4, 0, 0, 17, 0, 9, 3]}{{x}^{5}y+{x}^{4}{y}^{2}+3\,{x}^{3}{y}^{3}+ \left( 4+\sqrt {2} \right) 
{x}^{2}{y}^{4}+ \left( 1+\sqrt {2} \right) x{y}^{5}+4\,\sqrt {2}{y}^{6
}+2\,{x}^{4}+4\,{x}^{3}y+ \left( 4+4\,\sqrt {2} \right) x{y}^{3}+
 \left( 2+2\,\sqrt {2} \right) {y}^{4}+{x}^{2}+ \left( 1+4\,\sqrt {2}
 \right) {y}^{2}+2
}{1890000}{[1, 1, 2]_5}{[0, -1, 0, 1, 0, 1, 1, 0, 0, 0, 1, 0, 0, 1, 0, 0, 1, 0, 0, 0, 0, 0]}
\PPMM{36}{36}{8A\sb{1} + A\sb{2}}{1}{[0, 0, 0, 3, 0, 4, 11, 2, 7, 4]}{{x}^{5}y+4\,{x}^{2}{y}^{4}+{x}^{5}+3\,{x}^{4}y+2\,{x}^{2}{y}^{3}+3\,{x
}^{4}+2\,{y}^{4}+2\,x{y}^{2}+2\,{y}^{3}+2\,{x}^{2}+3\,xy+4\,y
}{3780000}{[1, 1, 1, 1, 1]_5}{[0, 0, 0, 0, 0, 0, 0, 0, 0, 1, 1, 0, 0, 0, 1, 0, 0, 0, 0, 1, 0, 1]}
\PPMM{37}{37}{8A\sb{1} + A\sb{2}}{1}{[0, 0, 0, 2, 0, 4, 13, 2, 6, 6]}{{x}^{4}{y}^{2}+4\,{x}^{3}{y}^{3}+4\,{x}^{2}{y}^{4}+3\,x{y}^{4}+{y}^{5}
+4\,x{y}^{3}+4\,{x}^{3}+4\,{x}^{2}y+4\,{x}^{2}+xy+3\,{y}^{2}+3\,x+3\,y
}{3024000}{[1, 1, 1, 1]_5}{[0, 0, 0, 0, 0, 0, 0, 0, 1, 0, 0, 0, 1, 0, 1, 1, 0, 0, 0, 0, 0, 1]}
\PPMM{38}{39}{8A\sb{1} + A\sb{2}}{1}{[0, 0, 0, 2, 0, 4, 12, 2, 10, 2]}{ \left( 1+4\,\sqrt {2} \right) {x}^{2}{y}^{4}+{x}^{4}y+ \left( 1+
\sqrt {2} \right) {x}^{3}{y}^{2}+3\,{x}^{2}{y}^{3}+ \left( 2+\sqrt {2}
 \right) x{y}^{4}+{x}^{4}+ \left( 2+2\,\sqrt {2} \right) {x}^{3}y+3\,{
x}^{2}{y}^{2}+\sqrt {2}{y}^{4}+4\,\sqrt {2}{x}^{3}+ \left( 2+3\,\sqrt 
{2} \right) {x}^{2}y+{y}^{3}+3\,{x}^{2}+ \left( 2+4\,\sqrt {2}
 \right) xy+3\,{y}^{2}
}{3024000}{[1, 1, 1, 1]_5}{[0, 0, 0, 0, 0, 0, 0, 1, 0, 0, 1, 0, 1, 0, 1, 0, 0, 0, 0, 0, 0, 1]}
\PPMM{40}{41}{6A\sb{1} + 2A\sb{2}}{6}{[0, 0, 0, 2, 0, 6, 6, 6, 6, 5]}{\sqrt {2}{x}^{6}+ \left( 1+\sqrt {2} \right) {x}^{5}y+ \left( 1+4\,
\sqrt {2} \right) {x}^{3}{y}^{3}+\sqrt {2}{x}^{2}{y}^{4}+2\,\sqrt {2}x
{y}^{5}+ \left( 3+\sqrt {2} \right) {y}^{6}+ \left( 4+3\,\sqrt {2}
 \right) {x}^{4}+3\,{x}^{3}y+ \left( 2+\sqrt {2} \right) {x}^{2}{y}^{2
}+ \left( 4+4\,\sqrt {2} \right) x{y}^{3}+ \left( 3+3\,\sqrt {2}
 \right) {y}^{4}+ \left( 2+\sqrt {2} \right) {x}^{2}+ \left( 1+4\,
\sqrt {2} \right) xy+\sqrt {2}{y}^{2}+4
}{378000}{[2]_5}{[0, 0, 0, 0, 0, 1, 0, 0, 0, 0, 0, 1, 1, 0, 0, 0, 1, 0, 0, 1, 0, 0]}
\PPMM{42}{43}{6A\sb{1} + 2A\sb{2}}{2}{[0, 0, 0, 2, 0, 6, 5, 6, 10, 1]}{{x}^{4}{y}^{2}+ \left( 2+3\,\sqrt {2} \right) {x}^{3}{y}^{3}+ \left( 3
+3\,\sqrt {2} \right) {x}^{2}{y}^{4}+{x}^{4}y+ \left( 3+3\,\sqrt {2}
 \right) {x}^{2}{y}^{3}+3\,\sqrt {2}x{y}^{4}+2\,\sqrt {2}{x}^{4}+
 \left( 3+4\,\sqrt {2} \right) {x}^{2}{y}^{2}+2\,\sqrt {2}x{y}^{3}+
 \left( 4+\sqrt {2} \right) {y}^{4}+ \left( 3+3\,\sqrt {2} \right) {x}
^{3}+ \left( 4+3\,\sqrt {2} \right) {y}^{3}+ \left( 4+2\,\sqrt {2}
 \right) {x}^{2}+4\,\sqrt {2}xy+ \left( 3+2\,\sqrt {2} \right) {y}^{2}
}{1512000}{[1, 1]_5}{[0, 0, 0, 0, 0, 0, 0, 0, 0, 1, 1, 1, 0, 0, 1, 0, 0, 0, 0, 0, 0, 1]}
\PPMM{44}{45}{6A\sb{1} + 2A\sb{2}}{1}{[0, 0, 0, 1, 0, 6, 7, 6, 9, 3]}{2\,{x}^{3}{y}^{3}+3\,\sqrt {2}{x}^{2}{y}^{4}+ \left( 4+2\,\sqrt {2}
 \right) {x}^{3}{y}^{2}+ \left( 3+\sqrt {2} \right) {x}^{2}{y}^{3}+
 \left( 1+2\,\sqrt {2} \right) x{y}^{4}+{x}^{4}+ \left( 3+\sqrt {2}
 \right) {x}^{3}y+3\,{x}^{2}{y}^{2}+3\,{y}^{4}+ \left( 1+4\,\sqrt {2}
 \right) {x}^{3}+ \left( 1+3\,\sqrt {2} \right) {x}^{2}y+4\,x{y}^{2}+
 \left( 2+4\,\sqrt {2} \right) {y}^{3}+ \left( 3+\sqrt {2} \right) {x}
^{2}+ \left( 1+\sqrt {2} \right) xy+ \left( 3+3\,\sqrt {2} \right) {y}
^{2}
}{2268000}{[1, 1, 1]_5}{[0, 0, 0, 0, 0, 0, 0, 0, 1, 0, 0, 0, 0, 1, 1, 0, 0, 1, 0, 0, 0, 1]}
\PPMM{46}{46}{4A\sb{1} + 3A\sb{2}}{3}{[0, 0, 0, 1, 0, 6, 3, 12, 7, 0]}{{x}^{6}+3\,{x}^{3}{y}^{3}+4\,{x}^{4}y+x{y}^{4}+3\,{x}^{2}{y}^{2}+4\,{x
}^{3}+3\,xy+4
}{756000}{[1]_5}{[0, 0, 0, 0, 0, 0, 0, 0, 0, 1, 0, 1, 1, 0, 1, 0, 0, 0, 0, 0, 0, 1]}
\PPMM{47}{47}{4A\sb{1} + 3A\sb{2}}{2}{[0, 0, 0, 0, 0, 8, 4, 8, 4, 5]}{{x}^{6}+3\,{x}^{4}{y}^{2}+4\,{x}^{2}{y}^{4}+2\,{y}^{6}+4\,{x}^{2}{y}^{
3}+2\,{x}^{4}+3\,{x}^{2}{y}^{2}+4\,{x}^{2}y+{y}^{3}+3\,{x}^{2}
}{1134000}{[1, 2]_5}{[0, 0, 0, 0, 0, 0, 0, 0, 1, 0, 0, 1, 1, 1, 1, 0, 0, 0, 0, 0, 0, 0]}
\PPMM{48}{48}{4A\sb{1} + 3A\sb{2}}{1}{[0, 0, 0, 0, 0, 8, 5, 8, 5, 4]}{2\,{x}^{4}{y}^{2}+{x}^{5}+2\,{x}^{2}{y}^{3}+4\,x{y}^{4}+2\,{x}^{3}y+3
\,{x}^{2}{y}^{2}+2\,x{y}^{3}+2\,x{y}^{2}+3\,{x}^{2}+2\,xy+2\,{y}^{2}
}{2268000}{[1, 1, 1]_5}{[0, 0, 0, 0, 0, 0, 0, 0, 0, 0, 0, 0, 0, 0, 1, 0, 1, 1, 0, 1, 0, 1]}
\PPMM{49}{49}{11A\sb{1}}{4}{[0, 0, 0, 8, 0, 0, 8, 0, 10, 0]}{{x}^{6}+{x}^{4}{y}^{2}+4\,{x}^{2}{y}^{4}+3\,{x}^{5}+3\,x{y}^{4}+{x}^{2
}{y}^{2}+2\,{y}^{4}+{x}^{3}+4\,{y}^{2}+2\,x+2
}{378000}{[2]_5}{[0, 0, 0, 0, 0, 0, 0, 1, 0, 0, 0, 1, 1, 0, 0, 1, 0, 1, 0, 0, 0, 0]}
\PPMM{50}{51}{9A\sb{1} + A\sb{2}}{1}{[0, 0, 0, 4, 0, 4, 9, 1, 7, 2]}{ \left( 4+\sqrt {2} \right) {x}^{3}{y}^{3}+ \left( 4+2\,\sqrt {2}
 \right) {x}^{2}{y}^{4}+{x}^{4}y+4\,x{y}^{4}+\sqrt {2}{x}^{4}+ \left( 
3+3\,\sqrt {2} \right) {x}^{2}{y}^{2}+4\,x{y}^{3}+ \left( 4+2\,\sqrt {
2} \right) {y}^{4}+ \left( 2+3\,\sqrt {2} \right) {x}^{3}+ \left( 4+4
\,\sqrt {2} \right) {x}^{2}y+ \left( 4+3\,\sqrt {2} \right) {y}^{3}+
 \left( 1+2\,\sqrt {2} \right) {x}^{2}+3\,\sqrt {2}xy+ \left( 2+3\,
\sqrt {2} \right) {y}^{2}
}{1512000}{[1, 1]_5}{[0, 0, 0, 0, 0, 0, 0, 0, 0, 1, 0, 1, 1, 0, 0, 0, 0, 1, 0, 0, 0, 1]}
\PPMM{52}{52}{7A\sb{1} + 2A\sb{2}}{2}{[0, 0, 0, 4, 0, 6, 2, 4, 8, 2]}{{x}^{6}+{x}^{5}y+2\,{x}^{4}{y}^{2}+{x}^{2}{y}^{4}+3\,{y}^{6}+{x}^{4}+{
x}^{2}{y}^{2}+x{y}^{3}+4\,xy+{y}^{2}+3
}{378000}{[2]_5}{[0, 0, 0, 0, 0, 0, 1, 1, 1, 0, 0, 1, 1, 0, 0, 0, 0, 0, 0, 0, 0, 0]}
\PPMM{53}{54}{7A\sb{1} + 2A\sb{2}}{1}{[0, 0, 0, 2, 0, 6, 7, 4, 7, 0]}{ \left( 2+2\,\sqrt {2} \right) {x}^{2}{y}^{4}+{x}^{4}y+ \left( 4+
\sqrt {2} \right) {x}^{3}{y}^{2}+ \left( 2+2\,\sqrt {2} \right) {x}^{2
}{y}^{3}+4\,x{y}^{4}+3\,{x}^{4}+ \left( 3+2\,\sqrt {2} \right) {x}^{3}
y+\sqrt {2}{x}^{2}{y}^{2}+ \left( 3+4\,\sqrt {2} \right) x{y}^{3}+
 \left( 2+2\,\sqrt {2} \right) {y}^{4}+ \left( 3+4\,\sqrt {2} \right) 
{x}^{2}y+ \left( 2+3\,\sqrt {2} \right) x{y}^{2}+ \left( 2+3\,\sqrt {2
} \right) {y}^{3}+\sqrt {2}xy+4\,{y}^{2}
}{1512000}{[1, 1]_5}{[0, 0, 0, 0, 0, 1, 0, 0, 0, 1, 0, 1, 0, 0, 1, 0, 0, 0, 0, 0, 0, 1]}
\PPMM{55}{56}{7A\sb{1} + 2A\sb{2}}{1}{[0, 0, 0, 2, 0, 6, 6, 4, 6, 1]}{ \left( 3+2\,\sqrt {2} \right) {x}^{2}{y}^{4}+{x}^{4}y+{x}^{3}{y}^{2}+
 \left( 2+\sqrt {2} \right) {x}^{2}{y}^{3}+ \left( 2+4\,\sqrt {2}
 \right) x{y}^{4}+\sqrt {2}{x}^{4}+ \left( 3+4\,\sqrt {2} \right) {x}^
{3}y+ \left( 2+4\,\sqrt {2} \right) x{y}^{3}+4\,\sqrt {2}{y}^{4}+
\sqrt {2}{x}^{3}+ \left( 1+4\,\sqrt {2} \right) {x}^{2}y+ \left( 4+
\sqrt {2} \right) {y}^{3}+4\,{x}^{2}+4\,\sqrt {2}xy+2\,{y}^{2}
}{1512000}{[1, 1]_5}{[0, 0, 0, 0, 0, 0, 0, 0, 0, 1, 0, 0, 0, 1, 1, 0, 1, 0, 0, 1, 0, 0]}
\PPMM{57}{58}{5A\sb{1} + 3A\sb{2}}{2}{[0, 0, 0, 1, 0, 6, 4, 9, 4, 1]}{\sqrt {2}{x}^{4}{y}^{2}+ \left( 4+2\,\sqrt {2} \right) {x}^{3}{y}^{3}+
4\,\sqrt {2}{x}^{2}{y}^{4}+ \left( 3+\sqrt {2} \right) x{y}^{5}+4\,
\sqrt {2}{y}^{6}+ \left( 1+4\,\sqrt {2} \right) {x}^{4}+ \left( 3+
\sqrt {2} \right) {x}^{3}y+2\,\sqrt {2}{x}^{2}{y}^{2}+ \left( 1+\sqrt 
{2} \right) {y}^{4}+ \left( 3+2\,\sqrt {2} \right) {x}^{2}+ \left( 2+4
\,\sqrt {2} \right) xy+ \left( 3+\sqrt {2} \right) {y}^{2}+1+4\,\sqrt 
{2}
}{756000}{[1]_5}{[0, 0, 0, 0, 0, 0, 0, 0, 0, 0, 1, 0, 1, 0, 0, 0, 0, 1, 0, 1, 0, 1]}
\PPMM{59}{60}{8A\sb{1} + 2A\sb{2}}{2}{[0, 0, 0, 2, 0, 6, 7, 2, 5, 0]}{{x}^{5}y+ \left( 1+\sqrt {2} \right) {x}^{4}{y}^{2}+2\,{x}^{2}{y}^{4}+
 \left( 2+\sqrt {2} \right) x{y}^{5}+ \left( 4+3\,\sqrt {2} \right) {y
}^{6}+3\,{x}^{4}+ \left( 4+4\,\sqrt {2} \right) {x}^{3}y+ \left( 1+3\,
\sqrt {2} \right) {x}^{2}{y}^{2}+ \left( 3+3\,\sqrt {2} \right) x{y}^{
3}+4\,\sqrt {2}{y}^{4}+4\,{x}^{2}+\sqrt {2}xy+ \left( 3+3\,\sqrt {2}
 \right) {y}^{2}+3
}{378000}{[2]_5}{[0, 0, 0, 0, 1, 0, 0, 1, 0, 0, 0, 0, 1, 0, 0, 1, 0, 0, 1, 0, 0, 0]}
\PPMM{61}{62}{8A\sb{1} + 2A\sb{2}}{1}{[0, 0, 0, 3, 0, 6, 7, 2, 3, 1]}{{x}^{3}{y}^{3}+ \left( 3+4\,\sqrt {2} \right) {x}^{2}{y}^{4}+2\,\sqrt 
{2}{x}^{3}{y}^{2}+2\,{x}^{2}{y}^{3}+2\,x{y}^{4}+{x}^{4}+ \left( 4+2\,
\sqrt {2} \right) {x}^{3}y+ \left( 1+\sqrt {2} \right) {x}^{2}{y}^{2}+
 \left( 3+2\,\sqrt {2} \right) {y}^{4}+ \left( 3+\sqrt {2} \right) {x}
^{3}+ \left( 3+4\,\sqrt {2} \right) {x}^{2}y+4\,\sqrt {2}{y}^{3}+
 \left( 2+4\,\sqrt {2} \right) {x}^{2}+ \left( 3+2\,\sqrt {2} \right) 
xy+ \left( 4+4\,\sqrt {2} \right) {y}^{2}
}{756000}{[1]_5}{[0, 0, 0, 0, 0, 0, 0, 1, 0, 1, 0, 1, 0, 0, 0, 0, 0, 1, 0, 0, 0, 1]}
\PPMM{63}{63}{6A\sb{1} + 3A\sb{2}}{3}{[0, 0, 0, 3, 0, 6, 3, 6, 3, 0]}{{x}^{4}{y}^{2}+{x}^{4}y+{x}^{3}{y}^{2}+2\,{y}^{5}+2\,{x}^{4}+4\,{x}^{2
}{y}^{2}+3\,x{y}^{3}+4\,{y}^{4}+4\,{x}^{3}+2\,x{y}^{2}+{y}^{3}+2\,{x}^
{2}+{y}^{2}
}{252000}{[3]_5}{[0, 0, 0, 1, 0, 0, 1, 0, 0, 0, 0, 0, 1, 0, 0, 1, 0, 0, 1, 0, 0, 0]}
\PPMM{64}{64}{6A\sb{1} + 3A\sb{2}}{3}{[0, 0, 0, 0, 0, 6, 9, 6, 0, 1]}{{x}^{4}{y}^{2}+{x}^{3}{y}^{3}+{x}^{2}{y}^{4}+3\,{x}^{3}{y}^{2}+{x}^{2}
{y}^{3}+3\,{x}^{3}y+{x}^{2}{y}^{2}+2\,{x}^{3}+2\,{x}^{2}y+3\,{y}^{3}+2
\,{x}^{2}+3\,xy+4\,y+4
}{252000}{[3]_5}{[0, -1, 0, 0, 1, 0, 1, 0, 1, 1, 0, 0, 0, 0, 0, 0, 1, 0, 0, 0, 0, 1]}
}
\section{Calculation of Example~\ref{example:nonprojautXF}}\label{sec:proof2}
The polynomials in Example~\ref{example:nonprojautXF}
that give a non-projective involution $g$ of $X_F$
are calculated by the following method.
Recall that $\theh\sprime$ in Step 8 of Section~\ref{sec:proof1}
is the representative vector of 
the  $\Aut(X_F, \theh)$-orbit  $\VVV_4\cap \projmodel_0$.
We have already calculated a birational morphism
$$
\phi_{\theh\sprime}=(\omega: \xi_0:\xi_1:\xi_2)\;:\; X_F\to X_{\theh\sprime},
$$
and the defining equation $s_{\theh\sprime}$ of $B_{\theh\sprime}$.
We have observed that $s_{\theh\sprime}$ is written as
$$
s_{\theh\sprime}(x,y,z)=\lambda^2\, \vx\, H\, \transpose{\overline{\vx}},
$$
where $\lambda\in \F_{25}\sptimes$, $\vx=(x,y,z)$, $\overline{\vx}=(x^5, y^5, z^5)$ 
and $H$ satisfies $H=\transpose{\overline{H}}$.
We search for $M\in \GL_3(\F_{25})$ such that $H=M\transpose{\overline{M}}$
(see~n.~3 of \cite{MR0213949}),
and put
$$
\omega\sprime:=\lambda\inv \omega,
\quad
(\xi_0\sprime,\xi_1\sprime,\xi_2\sprime):=(\xi_0,\xi_1,\xi_2)\,M.
$$
Then the polynomials $\omega\sprime, \xi_0\sprime,\xi_1\sprime,\xi_2\sprime$ satisfy
$$
\omega\sp{\prime 2}=\xi_0\sp{\prime 6}+\xi_1\sp{\prime 6}+\xi_2\sp{\prime 6}.
$$
Hence the rational map
from $X_F$ to $\P(3,1,1,1)$ given by
$(\omega\sprime: \xi_0\sprime: \xi_1\sprime: \xi_2\sprime)$
defines an automorphism $\gamma$ of $X_F$.
We choose $\theh$-lines $\ell_{i_1}, \dots, \ell_{i_{22}}$
such that $[\ell_{i_1}], \dots, [\ell_{i_{22}}]$ span $\NS(X)\tensor \Q$,
and that none of $i_1, \dots, i_{22}$ is contained in
the set $J$ of indices in the expression~\eqref{eq:dv} for $\theh\sprime$
that was used in the calculation of $\phi_{\theh\sprime}$.
Then we can calculate the images $\ell_{i_{\nu}}^\gamma$
of $\ell_{i_{\nu}}$ by $\gamma$ using the parametric representations of $\ell_{i_{\nu}}$
and the polynomials $(\omega\sprime: \xi_0\sprime: \xi_1\sprime: \xi_2\sprime)$.
Computing the intersection numbers of $\ell_{i_{\nu}}^\gamma$ 
with $\ell_1, \dots, \ell_{22}$,
we  calculate the action of $\gamma$ on $\NS(X)$.
Let $v\mapsto v\Gamma$ denote the matrix representation of this action.
We then search for $\tau\in \Aut(X, \theh)$ such that
its action on $X_F$ is given by
$$
w\mapsto \sigma w,\quad (x,y,z)\mapsto  (x,y,z)\,T_{\tau},
$$
where $\sigma\in \F_{25}\sptimes$, $T_{\tau}\in \GU_3(\F_{25})$, and 
its action on $\NS(X)$ is given by  $v\mapsto v N_{\tau}$,
where  $N_{\tau}$ is a matrix 
satisfying  $(\Gamma N_{\tau})^2=\Id_{22}$.
We define $(\omega\spprime, \xi_0\spprime,\xi_1\spprime,\xi_2\spprime)$ by
$$
\omega\spprime:=\sigma \omega\sprime,
\quad
(\xi_0\spprime,\xi_1\spprime,\xi_2\spprime):=(\xi_0\sprime,\xi_1\sprime,\xi_2\sprime)\, T_{\tau},
$$
and replace the original polynomials $(\omega, \xi_0, \xi_1, \xi_2)$ by  
$(\omega\spprime, \xi_0\spprime, \xi_1\spprime, \xi_2\spprime)$.
Then the automorphism $X_F\to X_F$ given by $(\omega:  \xi_0: \xi_1: \xi_2)$
is of order $2$,
because its action $v\mapsto v\Gamma N_{\tau}$ on $\NS(X)$ is of order $2$. 
%
%
\par
\medskip
{\bf Addendum.}
After the first version of this paper was finished,
we have investigated $X$ by Borcherds' method~\cite{MR913200, MR1654763},
and obtained polarizations $h_1$ and $h_2$ of degree $h_1^2=60$ and $h_2^2=80$
with large projective automorphism groups.
The  group $\Aut(X, h_1)$ is  isomorphic to the alternating group of degree $8$,
while the order of $\Aut(X, h_2)$ is $1152$.
%

%
%
%
\bibliographystyle{plain}

\def\cftil#1{\ifmmode\setbox7\hbox{$\accent"5E#1$}\else
  \setbox7\hbox{\accent"5E#1}\penalty 10000\relax\fi\raise 1\ht7
  \hbox{\lower1.15ex\hbox to 1\wd7{\hss\accent"7E\hss}}\penalty 10000
  \hskip-1\wd7\penalty 10000\box7} \def\cprime{$'$} \def\cprime{$'$}
  \def\cprime{$'$} \def\cprime{$'$}

\end{document}